\documentclass[10pt]{article}
\usepackage{geometry}
\usepackage{amssymb}
\usepackage{amsmath}
\usepackage{amsthm}
\usepackage{mathrsfs}
\usepackage{indentfirst}

\newtheorem{proposition}{Proposition}[section]
\newtheorem{lemma}{Lemma}[section]
\newtheorem{theorem}{Theorem}[section]
\newtheorem{corollary}{Corollary}[section]

\newtheorem{conjecture}{Conjecture}[section]

\newcommand{\dd}{\mathrm{d}}

\title{Shifted second moment of Gaussian Hecke $L$-functions $L(s,\lambda^k)$}
\author{Xin Hang Ji$^{1,2} $\\
	$ ^1$Department of Mathematics, ShangHai Normal University\\
	$ ^2$ Corresponding author. Email:
	\texttt{1000535190@smail.shnu.edu.cn}}

\begin{document}
	\maketitle
	
	\begin{abstract}
		We establish a uniform asymptotic formula for the smoothly weighted
		shifted second moment of the Gaussian angular Hecke $L$-functions
		$L(s,\lambda^k)$. The completed moment is expressed as the sum of four
		explicit main terms, corresponding to the four functional-equation
		swaps, with an error of size $O_{\Phi,\epsilon}(K^{1/2+\epsilon})$.
		After the Archimedean factors are removed, the resulting formula agrees
		with the numerator-only specialization of the four-swap prediction of
		the $L$-functions Ratios Conjecture. The proof transforms the
		off-diagonal contribution into Weyl sums over the roots of
		$r^2\equiv-1\bmod C$, realizes these sums spectrally through incomplete
		Poincar\'e series evaluated at $i$, and separates the Eisenstein and
		Maa\ss{} spectra. The two $v$-type main terms arise respectively from
		the zero frequency and from the combined residues of two moving
		Eisenstein poles, while the cuspidal spectrum is absorbed into the
		square-root error term.
		
	\end{abstract}
	
\section{Introduction}

The angular distribution of ideals and primes in the Gaussian field is
encoded by a natural family of Hecke characters of infinite order. This
family goes back to Hecke's work on zeta functions associated with
Gr\"ossencharakters and the equidistribution of prime ideals in sectors
\cite{Hecke}. Its $k$-th member is the $4k$-th angular Fourier mode on
$\mathbb Z[i]$; the finite conductor is fixed, while the Archimedean
conductor grows with $k$.

Let $\lambda$ be the Gr\"ossencharakter over the Gaussian field $\mathbb Q(i)$. For any nonzero Gaussian integer $z$, define
$\lambda((z))=z^4/|z|^4$. Since the fourth power of every Gaussian unit is $1$, this definition is independent of the choice of generator of the principal ideal. Let
$\mathrm{Id}^{*}$ denote the set of nonzero integral ideals of $\mathbb Z[i]$, and define
\[
L_k(s)=L(s,\lambda^k)
=\sum_{\mathfrak a\in\mathrm{Id}^{*}}
\frac{\lambda(\mathfrak a)^k}{N(\mathfrak a)^s}.
\]
Throughout, we sum only over positive integers $k$. Define the completed function
\begin{align*}
	\Lambda_k(s)
	&=\pi^{-s+\frac12}\Gamma(2k+s)L_k(s),
	\\
	\Lambda_k^*(s)
	&=\frac{\Lambda_k(s)}{\Gamma(2k+\frac12)}
	=\frac{\pi^{-s+\frac12}\Gamma(2k+s)L_k(s)}
	{\Gamma(2k+\frac12)}.
	\tag*{(1.1)}
\end{align*}
When $k\geq1$, $\Lambda_k^*(s)$ is entire and satisfies
$\Lambda_k^*(\frac12+s)=\Lambda_k^*(\frac12-s)$; see, for example,
\cite[Section 2]{Watt} for the same family and
\cite[Chapter 5]{IwaniecKowalski} for the general theory of Hecke
$L$-functions. The analytic conductor at the central point is of order
$k^2$. Thus the family considered here is a sparse CM family in the
Archimedean-conductor aspect, with one $L$-function for each positive $k$.

Earlier work provides several chronological benchmarks. In 1985 Sarnak
proved a sharp hybrid fourth-moment bound in the angular and spectral
parameters \cite{Sarnak}. In 2002 Conrey and Soundararajan evaluated a
shifted mollified second moment of quadratic Dirichlet $L$-functions and
showed that at least $20\%$ of the characters $\chi_{-8d}$ have no real
zero on $[0,1]$ \cite{ConreySoundararajan}. In 2012 Hough obtained a
four-term asymptotic for the harmonic twisted second moment of level-one
holomorphic cusp-form $L$-functions in the weight aspect, using the
Petersson trace formula and Voronoi summation \cite{Hough}. In 2013 Conrey
and Snaith treated powers of a CM Gr\"ossencharakter over
$\mathbb Q(\sqrt{-7})$, proving a first-moment asymptotic and a
second-moment upper bound and conjecturing a power-saving second-moment
formula \cite{ConreySnaith}. In 2014 Watt obtained weighted and twisted
extensions of Sarnak's hybrid estimates \cite{Watt}.

Work closer to the present setting developed next. In 2019 Rudnick and
Waxman studied Gaussian primes in shrinking angular sectors and formulated
a random-matrix prediction for the variance \cite{RudnickWaxman}. In 2021
Gao and Zhao obtained moment and one-level-density results for
finite-conductor Hecke families of trivial infinite type \cite{GaoZhao},
while Waxman computed lower-order terms in the one-level density of the
present symplectic family \cite{Waxman}. In 2023 Chen, Kim, Lichtman,
Miller, Shubina, Sweitzer, Waxman, Winsor, and Yang applied the Ratios
Conjecture to the Gaussian angular family, predicting the four-swap
structure and a square-root error \cite{ChenEtAl}. In 2024 J\"arviniemi
and Ter\"av\"ainen emphasized that even a Lindel\"of-order fourth moment
in the pure $k$-aspect is unknown \cite{JarviniemiTeravainen}, while David,
de Faveri, Dunn, and Stucky evaluated a mollified cubic second moment with
a power-saving error \cite{DavidDeFaveriDunnStucky}. In 2026 David, Devin,
and Waxman studied related CM elliptic-curve families
\cite{DavidDevinWaxman}; Castillo, de Faveri, and Dunn obtained the quartic
analogue, including quartic-twist families \cite{CastilloDeFaveriDunn};
and Diaconu, Ion, Pa\c{s}ol, and Popa established first and second twisted
moment asymptotics for general $r$-th order characters
\cite{DiaconuIonPasolPopa}.

These comparisons also explain why the present problem is especially
delicate. Conductor-aspect families admit Poisson summation and Gauss-sum
orthogonality, while Hough's complete cusp-form family admits the Petersson
trace formula. Here the finite conductor is fixed and there is only one
$L(s,\lambda^k)$ at each infinite type, so the $k$-average produces an
angular near-diagonal rather than an arithmetic delta symbol. Its
off-diagonal becomes a Weyl sum over $r^2\equiv-1\bmod C$ and must be
spectralized through Poincar\'e series at $i$. Furthermore, one $v$-type
main term comes from the zero frequency, whereas the other appears only as
the combined residue of two moving Eisenstein poles in a region where the
intermediate $q$-series is not absolutely convergent. Thus the full
four-swap asymptotic requires both spectral continuation and square-root
control of the remaining spectrum, making the sparse pure angular case the
most delicate of the comparisons above.

The purpose of this paper is to prove an unconditional shifted
second-moment formula for the Gaussian family, uniformly for small complex
shifts. The formula contains all four functional-equation swaps and has an
error $O_{\Phi,\epsilon}(K^{1/2+\epsilon})$, which is of square-root size
in the number of members of the family. To the best of the author's
knowledge, this is the first unconditional four-swap shifted second-moment
asymptotic in the pure angular aspect for this family.

A separate methodological inspiration is Motohashi's explicit formula for
the fourth power mean of the Riemann zeta-function
\cite{Motohashi1993}. Motohashi's conversion of an additive-divisor problem
into spectral data through the Kuznetsov trace formula suggested the
author's first approach: to force the present off-diagonal directly into a
Kuznetsov formula. That attempt was unsuccessful. After Poisson
summation the natural arithmetic object is $R(n;C)=
\sum_{\substack{r\bmod C\\r^2\equiv-1\bmod C}}
e\left(\frac{nr}{C}\right)$. This obstruction led to the
construction of Lemma 2.2. Instead of first converting $R(n;C)$ into
Kloosterman sums, we realize it as the value at $i$ of an incomplete
Poincar\'e series and spectrally expand that series itself. In this sense,
Lemma 2.2 is the replacement, adapted to the Gaussian angular family, for
the direct Kuznetsov step in the initial strategy.

The theorem is deliberately restricted to the untwisted moment. For a fixed
nonzero integral ideal $\mathfrak l$, Section 9 formulates a four-swap
conjecture for the moment with the additional factor
$\lambda(\mathfrak l)^k$. Although the diagonal Euler product is then
modified by only finitely many local factors, the proof is not a formal
variant of the argument below. The twist changes the congruence conditions
and arithmetic weights in the off-diagonal parametrization, so Lemma 2.2 no
longer applies in its present form: one would need a twist-dependent
incomplete Poincar\'e series and a correspondingly different spectral
expansion. Keeping track of its Eisenstein residues, continuous spectrum,
and Maa\ss{} coefficients would make the analysis prohibitively intricate.
For this reason, the present paper proves only the untwisted formula and
leaves the twisted case as the conjecture stated in Section 9.

Let $\Phi\in C_c^\infty((0,\infty))$, with $\operatorname{supp}\Phi\subset[1,2]$. Write
$\widetilde\Phi(s)=\int_1^2x^s\Phi(x)\,\dd x$, and define
\[
\mathcal S(u,v,K,\Phi)
=\sum_{k\geq1}\Phi\left(\frac{k}{K}\right)
\Lambda_k^*\left(\frac12+u+v\right)
\Lambda_k^*\left(\frac12+u-v\right).
\]

\begin{theorem}
	For every $\epsilon>0$, the following holds. Suppose that $K$ is sufficiently large,
	that $|\Re u|+|\Re v|\ll(\log K)^{-1}$, and that
	$|u|,|v|<\log K$, with
	$|u|,|v|,|u+v|,|u-v|\geq(\log K)^{-2}$. Then
	\begin{align*}
		\mathcal S(u,v,K,\Phi)
		={}&K\left(\frac{2K}{\pi}\right)^{2u}\widetilde\Phi(2u)
		\frac{\zeta(1+2v+2u)\zeta(1-2v+2u)L_0(1+2u)}
		{\zeta(2+4u)}
		\\
		&+K\left(\frac{2K}{\pi}\right)^{-2u}\widetilde\Phi(-2u)
		\frac{\zeta(1+2v-2u)\zeta(1-2v-2u)L_0(1-2u)}
		{\zeta(2-4u)}
		\\
		&+K\left(\frac{2K}{\pi}\right)^{2v}\widetilde\Phi(2v)
		\frac{\zeta(1+2u+2v)\zeta(1-2u+2v)L_0(1+2v)}
		{\zeta(2+4v)}
		\\
		&+K\left(\frac{2K}{\pi}\right)^{-2v}\widetilde\Phi(-2v)
		\frac{\zeta(1+2u-2v)\zeta(1-2u-2v)L_0(1-2v)}
		{\zeta(2-4v)}
		\\
		&+O_{\Phi,\epsilon}\left(K^{\frac12+\epsilon}\right).
		\tag*{(1.2)}
	\end{align*}
	Here $L_0(s)=\zeta(s)L(s,\chi_{-4})$.
\end{theorem}

The moment and the asserted main term are even in $v$. Indeed, replacing
$v$ by $-v$ merely interchanges the two completed $L$-factors, so
$\mathcal S(u,-v,K,\Phi)=\mathcal S(u,v,K,\Phi)$. On the right-hand side
of (1.2), the first two terms are unchanged, while the third and fourth
terms are interchanged; the shift hypotheses are invariant as well.
Consequently, it is enough to prove (1.2) when $\Re v>0$. The case
$\Re v<0$ follows by applying that result with $-v$ in place of $v$, and
the boundary case $\Re v=0$ follows by continuity from the uniform
estimate. Some intermediate continuation statements below are formulated
on larger strips for convenience.

\begin{corollary}
	Let $\alpha=u+v$ and $\beta=u-v$, and suppose that $u,v$ satisfy the shift conditions in the theorem. Then
	\begin{align*}
		&\sum_{k\geq1}\Phi\left(\frac{k}{K}\right)
		L_k\left(\frac12+\alpha\right)
		L_k\left(\frac12+\beta\right)
		\\
		={}&K\widetilde\Phi(0)
		\frac{\zeta(1+2\alpha)\zeta(1+2\beta)
			L_0(1+\alpha+\beta)}
		{\zeta(2+2\alpha+2\beta)}
		\\
		&+K\left(\frac{2K}{\pi}\right)^{-2\alpha-2\beta}
		\widetilde\Phi(-2\alpha-2\beta)
		\frac{\zeta(1-2\alpha)\zeta(1-2\beta)
			L_0(1-\alpha-\beta)}
		{\zeta(2-2\alpha-2\beta)}
		\\
		&+K\left(\frac{2K}{\pi}\right)^{-2\beta}
		\widetilde\Phi(-2\beta)
		\frac{\zeta(1+2\alpha)\zeta(1-2\beta)
			L_0(1+\alpha-\beta)}
		{\zeta(2+2\alpha-2\beta)}
		\\
		&+K\left(\frac{2K}{\pi}\right)^{-2\alpha}
		\widetilde\Phi(-2\alpha)
		\frac{\zeta(1-2\alpha)\zeta(1+2\beta)
			L_0(1-\alpha+\beta)}
		{\zeta(2-2\alpha+2\beta)}
		\\
		&+O_{\Phi,\epsilon}\left(K^{\frac12+\epsilon}\right).
		\tag*{(1.3)}
	\end{align*}
\end{corollary}

\begin{proof}
	By (1.1) and the uniform form of Stirling's formula, for $K\leq k\leq2K$ we have
	\begin{align*}
		&\pi^{\alpha+\beta}
		\frac{\Gamma(2k+\frac12)^2}
		{\Gamma(2k+\frac12+\alpha)
			\Gamma(2k+\frac12+\beta)}
		\\
		&\qquad=\left(\frac{\pi}{2k}\right)^{\alpha+\beta}
		\left(1+O_\epsilon(K^{-1+\epsilon})\right).
		\tag*{(1.4)}
	\end{align*}
	The error term in (1.4), as well as its $k$-derivatives of any fixed order, satisfies the corresponding uniform estimates under the shift bounds above. Therefore, the error term may be absorbed into the test weight, after which the uniform form of the theorem may be applied. Let
	$\Psi(x)=x^{-\alpha-\beta}\Phi(x)$. Thus
	\begin{align*}
		&\sum_{k\geq1}\Phi\left(\frac{k}{K}\right)
		L_k\left(\frac12+\alpha\right)
		L_k\left(\frac12+\beta\right)
		\\
		&\qquad=\left(\frac{\pi}{2K}\right)^{\alpha+\beta}
		\mathcal S\left(
		\frac{\alpha+\beta}{2},
		\frac{\alpha-\beta}{2},K,\Psi
		\right)
		+O_{\Phi,\epsilon}\left(K^{\frac12+\epsilon}\right).
		\tag*{(1.5)}
	\end{align*}
	Moreover, $\widetilde\Psi(s)=\widetilde\Phi(s-\alpha-\beta)$. Substituting
	$2u=\alpha+\beta$ and $2v=\alpha-\beta$ into (1.2) gives (1.3).
\end{proof}

Formula (1.3) is the numerator-only realization of the four-swap structure
predicted for this family by the ratios recipe in
\cite[Conjecture 4.1]{ChenEtAl}. For example, the no-swap Euler factor is
\[
\frac{\zeta(1+2\alpha)\zeta(1+2\beta)
	L_0(1+\alpha+\beta)}
{\zeta(2+2\alpha+2\beta)},
\]
and the other three terms are obtained by applying the functional equation
in the first variable, in the second variable, or in both variables,
together with the corresponding powers of $2K/\pi$. Thus the four-term
structure itself is conjecturally visible in the earlier ratios
calculation; the content of the theorem is its unconditional proof,
including the uniform square-root error.

We briefly describe the proof. A symmetric approximate functional equation
reduces the moment to a weighted sum over pairs of Gaussian integers.
Poisson summation in the angular parameter separates a diagonal
contribution from an off-diagonal contribution. After a sequence of
elementary parametrizations, the latter is expressed in terms of the sums
\[
R(n;C)=
\sum_{\substack{r\bmod C\\r^2\equiv-1\bmod C}}
e\left(\frac{nr}{C}\right).
\]
These sums are realized as the value at the elliptic point $i$ of an
incomplete Poincar\'e series. The standard spectral decomposition on
$\mathrm{SL}_2(\mathbb Z)\backslash\mathbb H$ then splits the nonzero
frequency into Eisenstein and Maa\ss{} contributions; background for the
spectral expansion and the local Weyl law may be found in
\cite{IwaniecSpectral}. The zero frequency supplies the main term with the
factor $(2K/\pi)^{2v}$. A common transform estimate supplies rapid decay in
the spectral parameter and summability in the remaining arithmetic
variables. It also gives normal convergence on compact subsets of the
shift region, which permits the initially truncated spectral identity to
be continued meromorphically. When the Eisenstein contour is moved, two
moving poles are crossed, and their combined residues supply the main term
with the factor $(2K/\pi)^{-2v}$ together with the term needed for the
final cancellation. The integral on the continued contour and the entire
Maa\ss{} spectrum contribute $O_{\Phi,\epsilon}(K^{1/2+\epsilon})$. No
unproved moment conjecture or spectral hypothesis is used.

The paper is organized as follows. Section 2 records the lattice identity
and the spectral expansion of the Weyl sums. Section 3 derives the
symmetric approximate functional equation. Section 4 establishes the
weight estimates and the diagonal--off-diagonal decomposition. The
diagonal is evaluated in Section 5, and the off-diagonal is parametrized in
Section 6. Section 7 treats its zero-frequency term. Section 8 proves the
common transform estimate, carries out the Eisenstein continuation and
residue calculation, bounds the Maa\ss{} spectrum, and combines the four
main terms. Finally, Section 9 formulates the twisted four-swap conjecture,
identifies its ideal-theoretic Euler factors, and explains the additional
spectral complications that prevent the present proof from treating the
twisted moment.

\subsection{Notation and Conventions}

Fix a final $\epsilon>0$. Throughout, we assume
$|\Re u|+|\Re v|\ll(\log K)^{-1}$ and
$|u|,|v|<\log K$, and
$|u|,|v|,|u+v|,|u-v|\geq(\log K)^{-2}$. The symbol $A$
denotes an arbitrarily large positive constant. Intermediate occurrences
of $K^\epsilon$ denote arbitrarily small auxiliary losses; since only
finitely many such losses occur, their sum is kept below the final
$\epsilon$ fixed above. Write $e(x)=e^{2\pi i x}$.
Every fixed power of $\log K$ arising from the shifts, the finitely many
differentiations, and the contour shifts under consideration is absorbed
into $K^\epsilon$.

Multiple series are summed from right to left. The order of summation is changed only after absolute convergence has been established or a truncation limit has been specified.
	\section{Two Basic Lemmas}
	
	\begin{lemma}
		When $\Re s>1$,
		\[
		L_k(s)=\frac14\sum_{m\in\mathbb Z[i]\setminus\{0\}}
		e^{4ik\arg m}|m|^{-2s}.
		\tag*{(2.1)}
		\]
		This identity is subsequently understood by analytic continuation.
	\end{lemma}
	
	\begin{proof}
		Every nonzero ideal is principal. A nonzero principal ideal has exactly four generators, which differ from one another by Gaussian units. Since the fourth power of every Gaussian unit is $1$, the four generators give the same value in the sum. Hence the sum over ideals is one-fourth of the sum over Gaussian integers.
	\end{proof}
	
	We next give the Weyl sum and spectral expansion needed below. We use the
	standard normalization of the spectral expansion in
	\cite[Chapters 3 and 7]{IwaniecSpectral}. Define
	\[
	R(n;C)=
	\sum_{\substack{r\bmod C\\r^2\equiv-1\bmod C}}
	e\left(\frac{nr}{C}\right).
	\]
	
	\begin{lemma}
		Let $n\in\mathbb Z\setminus\{0\}$ and
		$g\in C_c^\infty(\mathbb R_{>0})$. Then
		\begin{align*}
			\sum_{C\geq1}R(n;C)g(C)
			={}&\frac12\sum_j\overline{\rho_j(n)}u_j(i)\mathcal K_ng(t_j)
			\\
			&+\frac{1}{8\pi}\int_{-\infty}^{\infty}
			\overline{\rho_t(n)}E\left(i,\frac12+it\right)
			\mathcal K_ng(t)\,\dd t,
			\tag*{(2.2)}
		\end{align*}
		where
		\[
		\mathcal K_ng(t)=\int_0^\infty g(C)C^{-1/2}
		K_{it}\left(\frac{2\pi|n|}{C}\right)\,\dd C.
		\]
		Here $\{u_j\}$ is a standard orthonormal basis in $L^2(\mathrm{SL}_2(\mathbb Z)\backslash\mathbb H)$ of weight $0$ Maa\ss{} cusp forms, with the Fourier expansion
		\[
		u_j(z)=\sum_{m\neq0}\rho_j(m)\sqrt y
		K_{it_j}(2\pi|m|y)e(mx),
		\qquad z=x+iy.
		\]
		The standard Eisenstein series is denoted by
		\[
		E(z,s)=\sum_{\gamma\in\Gamma_\infty\backslash\Gamma}
		\Im(\gamma z)^s,
		\qquad \Gamma=\mathrm{SL}_2(\mathbb Z),
		\]
		where
		$\Gamma_\infty=\{\pm\left(\begin{smallmatrix}1&m\\0&1\end{smallmatrix}\right):m\in\mathbb Z\}$. At $s=\frac12+it$, its nonzero Fourier coefficients are
		\[
		\rho_t(m)=
		\frac{2|m|^{it}\sigma_{-2it}(|m|)}{\xi(1+2it)},
		\qquad
		\xi(s)=\pi^{-s/2}\Gamma\left(\frac{s}{2}\right)\zeta(s).
		\]
	\end{lemma}
	
	\begin{proof}
		Define the incomplete Poincar\'e series
		\[
		P_{n,g}(z)=
		\sum_{\gamma\in\Gamma_\infty\backslash\Gamma}
		g\left(\frac{1}{\Im(\gamma z)}\right)
		e\bigl(n\Re(\gamma z)\bigr).
		\]
		It is a smooth square-integrable automorphic function. If
		$\gamma=\left(\begin{smallmatrix}a&b\\c&d\end{smallmatrix}\right)$, then
		\[
		\gamma i=\frac{ac+bd+i}{c^2+d^2}.
		\]
		Let $C=c^2+d^2$ and $r=ac+bd$. The identity
		$r^2+1=(a^2+b^2)C$ gives $r^2\equiv-1\bmod C$.
		
		Write
		$\Gamma_i=\operatorname{Stab}_\Gamma(i)=\{\pm I,\pm S\}$, where
		$S=\left(\begin{smallmatrix}0&-1\\1&0\end{smallmatrix}\right)$. The map
		\[
		\Gamma_\infty\backslash\Gamma/\Gamma_i
		\longrightarrow
		\{(C,r):C\geq1,\ r\bmod C,\ r^2\equiv-1\bmod C\},
		\]
		$\gamma\mapsto(c^2+d^2,ac+bd\bmod C)$ is a bijection. To prove surjectivity, given $C$ and $r$, consider
		\[
		\psi_{C,r}:\mathbb Z[i]\longrightarrow\mathbb Z/C\mathbb Z,
		\qquad x+iy\longmapsto x-ry\bmod C.
		\]
		Its kernel is $(C,r+i)$, which has norm $C$. Choose a generator $c+id$; then
		$r+i=(a-ib)(c+id)$ gives $ac+bd=r$ and $ad-bc=1$. Injectivity follows from
		$(C,r+i)=(c+id)$, since the generators of a given ideal differ only by units, and right multiplication by $\Gamma_i$ realizes exactly these unit transformations.
		
		Each double coset contains two left cosets. Indeed,
		$\Gamma_i\cap\gamma^{-1}\Gamma_\infty\gamma=\{\pm I\}$. This is because the trace of $\pm S$ is $0$, whereas the noncentral elements of $\gamma^{-1}\Gamma_\infty\gamma$ have trace $\pm2$. The two left cosets in the same double coset give the same value at $i$, so
		\[
		P_{n,g}(i)=2\sum_{C\geq1}R(n;C)g(C).
		\tag*{(2.3)}
		\]
		
		Since $n\neq0$, the inner product with the constant spectrum is $0$. Unfolding the inner product with the cusp form and setting $C=y^{-1}$ gives
		\begin{align*}
			\langle P_{n,g},u_j\rangle
			&=\overline{\rho_j(n)}
			\int_0^\infty g(y^{-1})y^{-3/2}
			K_{it_j}(2\pi|n|y)\,\dd y
			\\
			&=\overline{\rho_j(n)}\mathcal K_ng(t_j).
			\tag*{(2.4)}
		\end{align*}
		Similarly,
		\[
		\left\langle P_{n,g},E\left(\cdot,\frac12+it\right)\right\rangle
		=\overline{\rho_t(n)}\mathcal K_ng(t).
		\]
		Substituting these inner products into the standard spectral expansion,
		evaluating at $z=i$, and then using (2.3) gives (2.2)
		\cite[Chapter 7]{IwaniecSpectral}.
	\end{proof}
	
	\begin{lemma}[Extension by truncation]
		Let $g\in C^\infty(\mathbb R_{>0})$. Suppose that there exists $\delta>0$ such that, for every $j\geq0$ and every $A>0$,
		\begin{align*}
			g^{(j)}(C)&\ll_{A,j}C^A,
			\qquad 0<C\leq1,
			\\
			g^{(j)}(C)&\ll_j C^{-1-\delta-j},
			\qquad C\geq1.
			\tag*{(2.5)}
		\end{align*}
		Then (2.2) remains valid, with the spectral side interpreted as the limit of smooth truncations. If the spectral side converges absolutely, this limit is the series and integral in the usual sense.
	\end{lemma}
	
	\begin{proof}
		Choose $\omega\in C_c^\infty([0,\infty))$ such that $\omega(x)=1$ for $0\leq x\leq1$, and let
		$g_X(C)=g(C)\omega(C/X)\omega(1/(XC))$. For each $X$, (2.2) may be applied. For any fixed $C>0$, one has $g_X(C)=g(C)$ when $X$ is sufficiently large. Condition (2.5) guarantees that the corresponding Poincar\'e series converge in every Sobolev norm of fixed order. Hence one may take limits in both the spectral projections and the point evaluation at $z=i$. The geometric side is controlled by
		$R(n;C)\ll_\epsilon C^\epsilon$ and (2.5). Letting $X\to\infty$ proves the result.
	\end{proof}
	
	\section{Approximate Functional Equation and Weight Function}
	
	Let $H(s)=e^{s^2}/s$ and $H_u(s)=H(s-u)+H(s+u)$. Take a vertical line to the right of all relevant poles, use the residue theorem to expand a contour around the origin into two vertical lines on the right and left, and then use the functional equation in the integral on the left. Since $H$ has Gaussian decay in the vertical direction, the integrals over the horizontal sides tend to $0$. Shifting the variables in the two integrals separately and moving the contours to $\Re s=1$ gives the exact identity
	\begin{align*}
		&\Lambda_k^*\left(\frac12+u+v\right)
		\Lambda_k^*\left(\frac12+u-v\right)
		\\
		&\qquad=\frac{1}{2\pi i}\int_{(1)}
		H_u(s)\Lambda_k^*\left(\frac12+v+s\right)
		\Lambda_k^*\left(\frac12-v+s\right)\,\dd s.
		\tag*{(3.1)}
	\end{align*}
	During the contour shifts, the poles $s=u$ and $s=-u$ of $H(s-u)$ and $H(s+u)$ always lie to the left of the new contour, while both completed functions are entire. In the subsequent expansion of the Dirichlet series, only $K\leq k\leq2K$ contributes; in this range, the arguments of all Gamma factors also remain uniformly away from their poles.
	
	Define
	\begin{align*}
		V_t(x)
		={}&\frac{1}{2\pi i}\int_{(1)}H_u(s)
		\frac{\Gamma(2t+\frac12+v+s)
			\Gamma(2t+\frac12-v+s)}
		{\Gamma(2t+\frac12)^2}
		\pi^{-2s}x^{-\frac12-s}\,\dd s.
		\tag*{(3.2)}
	\end{align*}
	Substitute (2.1) into (3.1). All Dirichlet series converge absolutely on $\Re s=1$, so
	\begin{align*}
		\mathcal S(u,v,K,\Phi)
		={}&\frac1{16}
		\sum_{m_1,m_2\in\mathbb Z[i]\setminus\{0\}}
		|m_1|^{-2v}|m_2|^{2v}
		\\
		&\qquad\times
		\sum_{k\geq1}\Phi\left(\frac{k}{K}\right)
		e^{4ik\arg(m_1m_2)}V_k(|m_1m_2|^2).
		\tag*{(3.3)}
	\end{align*}
	Let
	\[
	W(z)=\sum_{k\geq1}\Phi\left(\frac{k}{K}\right)
	e^{4ik\arg z}V_k(|z|^2),
	\qquad z\neq0.
	\]
	Thus
	\[
	\mathcal S(u,v,K,\Phi)
	=\frac1{16}
	\sum_{m_1,m_2\in\mathbb Z[i]\setminus\{0\}}|m_1|^{-2v}|m_2|^{2v}W(m_1m_2).
	\tag*{(3.4)}
	\]
	
	\section{Estimates for $W$ and the Main Decomposition}
	
	All arguments are taken in $[-\pi,\pi)$. We first record uniform derivative estimates for $V_t$.
	
	\begin{lemma}
		Let $\sigma\geq1$ be fixed, and let $j\geq0$. For $K\leq t\leq2K$ and $z\neq0$,
		\[
		\frac{\partial^j}{\partial t^j}V_t(|z|^2)
		\ll_{\sigma,j,\epsilon}
		K^{2\sigma-j+\epsilon}|z|^{-1-2\sigma}.
		\tag*{(4.1)}
		\]
	\end{lemma}
	
	\begin{proof}
		Move the contour in (3.2) to $\Re s=\sigma$. The poles $s=\pm u$ lie to the left of the contour and hence are not crossed. On the new contour, apply the uniform form of Stirling's formula to the two Gamma quotients. Each differentiation with respect to $t$ produces a factor $K^{-1+\epsilon}$. Moreover,
		$H_u(\sigma+i\tau)$ has Gaussian decay near $\tau=\pm\Im u$, respectively, so all fixed-order moments of the $\tau$-integral are uniformly controlled. Finally, the Mellin factor gives $|z|^{-1-2\sigma}$. Combining these estimates yields (4.1).
	\end{proof}
	
	\begin{lemma}
		When $|z|>K^{1+\epsilon}$, for every $A>0$,
		\[
		W(z)\ll_{\Phi,\epsilon,A}|z|^{-A}.
		\tag*{(4.2)}
		\]
	\end{lemma}
	
	\begin{proof}
		By the case $j=0$ of (4.1),
		\[
		W(z)\ll_{\Phi,\sigma,\epsilon}
		K^{2\sigma+1+\epsilon}|z|^{-1-2\sigma}.
		\tag*{(4.3)}
		\]
		Given $A$, first choose $2\sigma+1>A$, and then take $\sigma$ sufficiently large that
		$\epsilon(2\sigma+1)>(1+\epsilon)A+2\epsilon$. When
		$|z|>K^{1+\epsilon}$, (4.3) then implies (4.2).
	\end{proof}
	
	Define the continuous weight
	\[
	W_1(z)=\int_K^{2K}\Phi\left(\frac{t}{K}\right)
	e^{4it\arg z}V_t(|z|^2)\,\dd t.
	\]
	
	\begin{lemma}
		When $K>100$, for every $A>0$,
		\[
		W_1(z)=
		\begin{cases}
			W(z)+O_{\Phi,\epsilon,A}(K^{-A});
			&-\frac\pi4\leq\arg z<\frac\pi4,\quad 1\le|z|\leq K^{1+\epsilon},\\
			O_{\Phi,\epsilon,A}(K^{-A});
			&\arg z\notin[-\frac\pi4,\frac\pi4),\quad 1\le|z|\leq K^{1+\epsilon},\\
			O_{\Phi,\epsilon,A}(|z|^{-A});
			&|z|>K^{1+\epsilon}.
		\end{cases}
		\tag*{(4.4)}
		\]
	\end{lemma}
	
	\begin{proof}
		Write $\theta=\arg z$ and
		$A_z(t)=\Phi(t/K)V_t(|z|^2)$, and extend $A_z$ to a smooth compactly supported function on the real line. By (4.1) and the Leibniz rule, for every $j\geq0$,
		\[
		\int_{\mathbb R}|A_z^{(j)}(t)|\,\dd t
		\ll_{\Phi,j,\epsilon}K^{3-j+\epsilon},
		\qquad 1\leq|z|\leq K^{1+\epsilon}.
		\tag*{(4.5)}
		\]
		Applying the Poisson summation formula over integers $k$ gives
		\begin{align*}
			W(z)&=\sum_{\ell\in\mathbb Z}I_\ell(z),
			\\
			I_\ell(z)&=\int_{\mathbb R}A_z(t)
			e^{i(4\theta-2\pi\ell)t}\,\dd t.
			\tag*{(4.6)}
		\end{align*}
		where $I_0(z)=W_1(z)$. If $4\theta-2\pi\ell\neq0$, integrating by parts $j$ times gives
		\[
		I_\ell(z)\ll_{\Phi,j,\epsilon}
		K^{3-j+\epsilon}|4\theta-2\pi\ell|^{-j}.
		\tag*{(4.7)}
		\]
		When $-\pi/4\leq\theta<\pi/4$ and $\ell\neq0$,
		$|4\theta-2\pi\ell|\geq\pi(2|\ell|-1)$. Summing over $\ell$ and taking $j$ sufficiently large gives
		$W(z)-W_1(z)\ll K^{-A}$. When $\theta$ does not lie in this angular interval,
		$|4\theta|\geq\pi$, so applying the same integration by parts to $I_0$ gives
		$W_1(z)\ll K^{-A}$.
		
		Finally, if $|z|>K^{1+\epsilon}$, then by (4.1) and the fact that the length of the integration interval is $K$,
		\[
		W_1(z)\ll_{\Phi,\sigma,\epsilon}
		K^{2\sigma+1+\epsilon}|z|^{-1-2\sigma}.
		\]
		Choosing $\sigma$ as in (4.2) gives the third case.
	\end{proof}
	
	\subsection{From $W$ to $(D)+(OD)$}
	
	Let
	$\mathfrak S=\{z\neq0:-\pi/4\leq\arg z<\pi/4\}$. For any Gaussian unit
	$\varepsilon\in\{\pm1,\pm i\}$, since $k$ is an integer, we have
	$W(\varepsilon z)=W(z)$. In each quadruple
	$(m_1,m_2)\mapsto(\varepsilon m_1,m_2)$, exactly one product lies in $\mathfrak S$. Therefore, (3.4) can be rewritten exactly as
	\[
	\mathcal S(u,v,K,\Phi)
	=\frac14
	\sum_{\substack{m_1,m_2\in\mathbb Z[i]\setminus\{0\}\\m_1m_2\in\mathfrak S}}
	|m_1|^{-2v}|m_2|^{2v}W(m_1m_2).
	\tag*{(4.8)}
	\]
	
	\begin{proposition}
		For every $A>0$,
		\begin{align*}
			\mathcal S(u,v,K,\Phi)
			={}&\frac14
			\sum_{\substack{m_1,m_2\in\mathbb Z[i]\setminus\{0\}\\m_1m_2\in\mathbb N^*}}
			|m_1|^{-2v}|m_2|^{2v}W_1(m_1m_2)
			\\
			&+\frac14
			\sum_{\substack{m_1,m_2\in\mathbb Z[i]\setminus\{0\}\\\Im(m_1m_2)\neq0}}
			|m_1|^{-2v}|m_2|^{2v}W_1(m_1m_2)
			+O_{\Phi,\epsilon,A}(K^{-A}).
			\tag*{(4.9)}
		\end{align*}
	\end{proposition}
	
	Denote the first and second terms on the right-hand side of (4.9) by diagonal term $(D)$ and off-diagonal term $(OD)$, respectively. More precisely,
	\begin{align*}
		(D)&=\sum_{n\geq1}r_v(n)W_1(n),
		\\
		r_v(n)&=\frac14
		\sum_{\substack{m_1,m_2\in\mathbb Z[i]\setminus\{0\}\\m_1m_2=n}}
		|m_1|^{-2v}|m_2|^{2v},
		\tag*{(4.10)}
	\end{align*}
	and
	\begin{align*}
		(OD)&=\frac14\sum_{n\in\mathbb Z\setminus\{0\}}S_n,
		\\
		S_n&=\sum_{\substack{m_1,m_2\in\mathbb Z[i]\setminus\{0\}\\\Im(m_1m_2)=n}}
		|m_1|^{-2v}|m_2|^{2v}W_1(m_1m_2).
		\tag*{(4.11)}
	\end{align*}
	
	\begin{proof}
		First replace $W$ by $W_1$ in (4.8). When $|m_1m_2|\leq K^{1+\epsilon}$, (4.4) gives an error $O(K^{-B})$ term by term. There are $O(K^{2+2\epsilon})$ Gaussian integers of norm at most $K^{1+\epsilon}$, the number of factorizations of each Gaussian integer is $O_\epsilon(|z|^\epsilon)$, and the absolute value of the weight is at most $|z|^{2|\Re v|}$. Taking $B$ sufficiently large, the total error from this part is $O(K^{-A})$. When $|m_1m_2|>K^{1+\epsilon}$, (4.2) and (4.4) give decay of arbitrary polynomial order, so the tail is also $O(K^{-A})$.
		
		Therefore, the restriction to the angular interval may be removed, giving
		\[
		\mathcal S(u,v,K,\Phi)
		=\frac14\sum_{m_1,m_2\in\mathbb Z[i]\setminus\{0\}}
		|m_1|^{-2v}|m_2|^{2v}W_1(m_1m_2)
		+O_{\Phi,\epsilon,A}(K^{-A}).
		\tag*{(4.12)}
		\]
		If $m_1m_2$ is a negative real number, then $\arg(m_1m_2)=-\pi$, so its contribution is absorbed into the error term by (4.4). The part on the positive real axis is exactly (4.10), and the part off the real axis is exactly (4.11). This proves (4.9).
	\end{proof}
	
	\section{Evaluation of $(D)$}
	
	Set
	\[
	\mathcal J(s)=\int_K^{2K}\Phi\left(\frac{t}{K}\right)
	\frac{\Gamma(2t+\frac12+v+s)
		\Gamma(2t+\frac12-v+s)}
	{\Gamma(2t+\frac12)^2}\,\dd t.
	\tag*{(5.1)}
	\]
	
	\begin{proposition}
		\begin{align*}
			(D)
			={}&K\left(\frac{2K}{\pi}\right)^{2u}\widetilde\Phi(2u)
			\frac{\zeta(1+2v+2u)\zeta(1-2v+2u)L_0(1+2u)}
			{\zeta(2+4u)}
			\\
			&+K\left(\frac{2K}{\pi}\right)^{-2u}\widetilde\Phi(-2u)
			\frac{\zeta(1+2v-2u)\zeta(1-2v-2u)L_0(1-2u)}
			{\zeta(2-4u)}
			\\
			&+\frac K2\left(\frac{2K}{\pi}\right)^{2v}
			\widetilde\Phi(2v)H_u(v)
			\frac{\zeta(1+4v)L_0(1+2v)}{\zeta(2+4v)}
			\\
			&+\frac K2\left(\frac{2K}{\pi}\right)^{-2v}
			\widetilde\Phi(-2v)H_u(-v)
			\frac{\zeta(1-4v)L_0(1-2v)}{\zeta(2-4v)}
			\\
			&+O_{\Phi,\epsilon}\left(K^{\frac12+\epsilon}\right).
			\tag*{(5.2)}
		\end{align*}
	\end{proposition}
	
	\begin{proof}
		By (3.2) and (4.10), interchanging the absolutely convergent sum and integral on $\Re s=1$ gives
		\[
		(D)=\frac{1}{2\pi i}\int_{(1)}
		\mathcal J(s)H_u(s)\pi^{-2s}
		\sum_{n\geq1}\frac{r_v(n)}{n^{1+2s}}\,\dd s.
		\tag*{(5.3)}
		\]
		The factor $1/4$ exactly cancels the fourfold counting caused by the Gaussian units, and hence
		\[
		r_v(n)=\sum_{\mathfrak a\mathfrak b=(n)}
		N(\mathfrak a)^{-v}N(\mathfrak b)^v.
		\]
		In particular, $r_v$ is multiplicative. For a prime $p$, let
		$X=p^{-1-2s}$ and $Y=p^{2v}$. Enumerating the ideal divisors according to the splitting type of $p$ in $\mathbb Q(i)$ gives
		\[
		r_v(p^\nu)=
		\begin{cases}
			\displaystyle\sum_{a,b=0}^{\nu}Y^{\nu-a-b},
			&p\equiv1\bmod4,\\
			\displaystyle\sum_{a=0}^{\nu}Y^{\nu-2a},
			&p\equiv3\bmod4,\\
			\displaystyle\sum_{a=0}^{2\nu}Y^{\nu-a},
			&p=2.
		\end{cases}
		\]
		All three local generating functions can be written as
		\[
		\sum_{\nu\geq0}r_v(p^\nu)X^\nu
		=\frac{1-X^2}
		{(1-XY)(1-XY^{-1})(1-X)(1-\chi_{-4}(p)X)}.
		\tag*{(5.4)}
		\]
		Multiplying over all primes, we obtain
		\[
		\sum_{n\geq1}\frac{r_v(n)}{n^{1+2s}}
		=\frac{\zeta(1+2v+2s)\zeta(1-2v+2s)L_0(1+2s)}
		{\zeta(2+4s)}.
		\tag*{(5.5)}
		\]
		
		Substitute (5.5) into (5.3), and then shift the contour of integration to $\Re s=-1/4$. In the strip traversed, the possible poles are at $s=\pm u$ and $s=\pm v$, while the pole at $s=0$ is canceled by $H_u(0)=0$.
		
		On $\Re s=-1/4$, Stirling's formula gives
		$\mathcal J(s)\ll_\Phi K^{1/2}(1+|s|)^B$. Combining the polynomial growth of the zeta function and the Dirichlet $L$-function in vertical strips with the Gaussian decay of $H_u$, we obtain
		\[
		\frac{1}{2\pi i}\int_{(-1/4)}
		\mathcal J(s)H_u(s)\pi^{-2s}
		\frac{\zeta(1+2v+2s)\zeta(1-2v+2s)L_0(1+2s)}
		{\zeta(2+4s)}\,\dd s
		\ll_{\Phi,\epsilon}K^{\frac12+\epsilon}.
		\tag*{(5.6)}
		\]
		For $s\in\{u,-u,v,-v\}$, the uniform version of Stirling's formula gives
		\[
		\mathcal J(s)
		=2^{2s}K^{1+2s}\widetilde\Phi(2s)
		+O_{\Phi,\epsilon}\left(K^{2\Re s+\epsilon}\right).
		\tag*{(5.7)}
		\]
		By the lower bounds on the shifts, the remaining factors in the four residues cost at most a fixed power of $\log K$. Computing the residues at $s=\pm u,\pm v$ directly and then using (5.7), the resulting total error is $O_{\Phi,\epsilon}(K^\epsilon)$. Combining this with (5.6) proves (5.2).
	\end{proof}
	\section{Decomposition of $(OD)$}
	
	By (4.11),
	\[
	(OD)=\frac14\sum_{n\neq0}S_n.
	\]
	The rapid decay from Section 4 and the divisor bound in the Gaussian integers show that this series converges absolutely. Indeed, if $z=m_1m_2$, then the absolute value of the weight is at most
	$|z|^{2|\Re v|}$, while the number of factorizations of $z$ in the Gaussian integers is
	$O_\delta(|z|^\delta)$.
	
	Let $m_1=a-ib$ and $m_2=c+id$. Then
	\[
	m_1m_2=ac+bd+i(ad-bc).
	\]
	Thus
	\[
	S_n=
	\sum_{\substack{a,b,c,d\in\mathbb Z\\ad-bc=n}}
	(a^2+b^2)^{-v}(c^2+d^2)^v
	W_1(ac+bd+in).
	\tag*{(6.1)}
	\]
	
	Let $h=(a,b)>0$, and write $a=h\alpha$ and $b=h\beta$, where
	$(\alpha,\beta)=1$. Since $ad-bc=n$, we have $h\mid|n|$. Set
	\[
	\eta=\frac{n}{h},
	\qquad C=\alpha^2+\beta^2,
	\qquad \ell=\alpha c+\beta d.
	\]
	Solving the simultaneous equations $\alpha d-\beta c=\eta$ and $\alpha c+\beta d=\ell$ gives
	\[c=\frac{\alpha\ell-\beta\eta}{C},\quad
	d=\frac{\beta\ell+\alpha\eta}{C}.
	\tag*{(6.2)}\]
	At the same time,
	\begin{align*}
		a^2+b^2&=h^2C,
		\\
		ac+bd&=h\ell,
		\\
		c^2+d^2&=\frac{\ell^2+\eta^2}{C}
		=\frac{h^2\ell^2+n^2}{h^2C}.
		\tag*{(6.3)}
	\end{align*}
	
	When $C>1$, the condition $(\alpha,\beta)=1$ implies
	$(\alpha,C)=(\beta,C)=1$. Define
	$r\equiv\beta\alpha^{-1}\bmod C$. Since
	$\alpha^2+\beta^2=C$, we have $r^2\equiv-1\bmod C$. By (6.2),
	\[
	c,d\in\mathbb Z
	\quad\Longleftrightarrow\quad
	\ell\equiv\eta r\bmod C.
	\tag*{(6.4)}
	\]
	The same conclusion also holds when $C=1$.
	
	We next check the multiplicity. Let
	\[
	\mathcal P(C)=
	\{(\alpha,\beta)\in\mathbb Z^2:
	(\alpha,\beta)=1,\ \alpha^2+\beta^2=C\}.
	\]
	The map $(\alpha,\beta)\mapsto r$ takes the same value on each orbit under the four unit rotations. Conversely, given
	$r^2\equiv-1\bmod C$, the ideal $I_r=(C,r-i)$ has norm $C$. Writing
	$I_r=(\alpha+i\beta)$ gives $\alpha^2+\beta^2=C$. Moreover,
	$r-i\in I_r$ ensures that $(\alpha,\beta)=1$. The generators of a given ideal differ by precisely a unit. Therefore
	\[
	\mathcal P(C)/\{\pm1,\pm i\}
	\longleftrightarrow
	\{r\bmod C:r^2\equiv-1\bmod C\}
	\tag*{(6.5)}
	\]
	is a bijection.
	
	For $n\neq0$, define
	\[
	F_n(x)=(x^2+n^2)^vW_1(x+in),
	\tag*{(6.6)}
	\]
	where complex powers of positive real numbers are defined using the real logarithm. By (6.3),
	\[
	(a^2+b^2)^{-v}(c^2+d^2)^vW_1(ac+bd+in)
	=h^{-4v}C^{-2v}F_n(h\ell).
	\tag*{(6.7)}
	\]
	Substituting (6.4), (6.5), and (6.7) into (6.1), we obtain
	\[
	S_n=4\sum_{h\mid|n|}h^{-4v}
	\sum_{C\geq1}C^{-2v}
	\sum_{\substack{r\bmod C\\r^2\equiv-1\bmod C}}
	\sum_{\ell\equiv\eta r\, (\bmod C)}F_n(h\ell),
	\qquad \eta=\frac{n}{h}.
	\tag*{(6.8)}
	\]
	
	By (3.2), (4.1), and repeated integration by parts with respect to $t$, for every
	$j,A\geq0$ there is a constant $B_{A,j}$ such that
	\[
	F_n^{(j)}(x)
	\ll_{A,j,\Phi,\epsilon}
	K^{B_{A,j}}(1+|x|+|n|)^{-A}.
	\tag*{(6.9)}
	\]
	Here we shift the Mellin contour when $|x+in|$ is large, and integrate by parts with respect to $t$ when the argument is bounded away from $0$. Thus $F_n$ is a Schwartz function. We use the Fourier transform convention
	\[
	\widehat F_n(y)=\int_{\mathbb R}F_n(x)e(-xy)\,\dd x.
	\]
	Applying Poisson summation on arithmetic progressions to $\ell$ in (6.8),
	\begin{align*}
		\sum_{\ell\equiv\eta r\, (\bmod C)}F_n(h\ell)
		&=\sum_{m\in\mathbb Z}F_n\bigl(h(\eta r+Cm)\bigr)
		\\
		&=\frac{1}{hC}\sum_{q\in\mathbb Z}
		e\left(\frac{q\eta r}{C}\right)
		\widehat F_n\left(\frac{q}{hC}\right).
		\tag*{(6.10)}
	\end{align*}
	we obtain
	\[
	S_n=4\sum_{h\mid|n|}h^{-1-4v}
	\sum_{C\geq1}\sum_{q\in\mathbb Z}
	\frac{R(q\eta;C)}{C^{1+2v}}
	\widehat F_n\left(\frac{q}{hC}\right).
	\tag*{(6.11)}
	\]
	
	Here, for fixed $n,h,C$, the $q$-series converges absolutely by (6.9). Summing over $q$ first, in this order, (6.10) recovers the sum over the arithmetic progression in (6.8). The latter is a reparametrization of the original absolutely convergent series (6.1), so the subsequent summation over $C$ is justified.
	
	Substituting (6.11) back into (4.11) and then using $n=h\eta$, we obtain
	\[
	(OD)=\sum_{h\geq1}\sum_{\eta\neq0}h^{-1-4v}
	\sum_{C\geq1}\sum_{q\in\mathbb Z}
	\frac{R(q\eta;C)}{C^{1+2v}}
	\widehat F_{h\eta}\left(\frac{q}{hC}\right).
	\tag*{(6.12)}
	\]
	
	By the $v\leftrightarrow -v$ symmetry recorded after (1.2), it is enough
	to work with $\Re v>0$. In this half-plane,
	$R(0;C)\ll_\delta C^\delta$, so choosing
	$0<\delta<2\Re v$ makes the $C$-series in the $q=0$ contribution absolutely convergent. Moreover, by (6.9),
	$\widehat F_n(0)$ decays rapidly with respect to $|n|$. We may therefore define
	\[
	(OD)_0=
	\sum_{h\geq1}\sum_{\eta\neq0}h^{-1-4v}
	\sum_{C\geq1}\frac{R(0;C)}{C^{1+2v}}
	\widehat F_{h\eta}(0),
	\tag*{(6.13)}
	\]
	and
	\[
	(OD)^*=\sum_{h\geq1}\sum_{\eta\neq0}h^{-1-4v}
	\sum_{C\geq1}\sum_{q\neq0}
	\frac{R(q\eta;C)}{C^{1+2v}}
	\widehat F_{h\eta}\left(\frac{q}{hC}\right).
	\tag*{(6.14)}
	\]
	When $\Re v>0$,
	\[
	(OD)=(OD)_0+(OD)^*.
	\tag*{(6.15)}
	\]
	Here (6.14) is to be interpreted with the summations taken in the order $q,C,\eta,h$. The continuation to the target region cannot be obtained merely by formal analytic continuation. One must first introduce a smooth cutoff in $C$ within the region of absolute convergence, then use the spectral expansion and the extension from truncated sums in Section 2, and finally prove uniform convergence on the spectral side. This is carried out in Section 8. Section 7 treats only the already absolutely convergent term $(OD)_0$.
	
	\section{Evaluation of $(OD)_0$}
	
	\begin{proposition}
		In the target range of shifts, the meromorphic continuation of $(OD)_0$ satisfies
		\begin{align*}
			(OD)_0
			={}&K\left(\frac{2K}{\pi}\right)^{2v}\widetilde\Phi(2v)
			\frac{L_0(1+2v)\zeta(1+2v+2u)\zeta(1+2v-2u)}
			{\zeta(2+4v)}
			\\
			&-\frac K2\left(\frac{2K}{\pi}\right)^{2v}
			\widetilde\Phi(2v)H_u(v)
			\frac{L_0(1+2v)\zeta(1+4v)}{\zeta(2+4v)}
			+O_{\Phi,\epsilon}(K^\epsilon).
			\tag*{(7.1)}
		\end{align*}
	\end{proposition}
	
	\begin{proof}
		We first compute the Dirichlet series of $R(0;C)$. For prime powers,
		\[
		R(0;p^a)=
		\begin{cases}
			2,&p\equiv1\bmod4,\quad a\geq1,\\
			0,&p\equiv3\bmod4,\quad a\geq1,\\
			1,&p=2,\quad a=1,\\
			0,&p=2,\quad a\geq2.
		\end{cases}
		\]
		Thus, when $\Re w>1$,
		\begin{align*}
			\sum_{C\geq1}\frac{R(0;C)}{C^w}
			&=\prod_p\frac{1-p^{-2w}}
			{(1-p^{-w})(1-\chi_{-4}(p)p^{-w})}
			\\
			&=\frac{\zeta(w)L(w,\chi_{-4})}{\zeta(2w)}
			=\frac{L_0(w)}{\zeta(2w)}.
			\tag*{(7.2)}
		\end{align*}
		Taking $w=1+2v$ and pairing the terms with $\eta>0$ and $\eta<0$, we obtain
		\[
		(OD)_0=
		\frac{L_0(1+2v)}{\zeta(2+4v)}
		\sum_{h\geq1}\sum_{\eta\geq1}h^{-1-4v}
		\left(\widehat F_{h\eta}(0)+\widehat F_{-h\eta}(0)\right).
		\tag*{(7.3)}
		\]
		
		We next compute $\widehat F_n(0)$. Shift the Mellin contour in (3.2) to $\Re s=2$ and then apply Fubini's theorem to obtain
		\begin{align*}
			\widehat F_n(0)
			={}&\frac{1}{2\pi i}\int_{(2)}H_u(s)\pi^{-2s}
			\int_K^{2K}\Phi\left(\frac{t}{K}\right)
			\frac{\Gamma(\frac12+2t+v+s)
				\Gamma(\frac12+2t-v+s)}
			{\Gamma(\frac12+2t)^2}
			\\
			&\qquad\times
			\int_{\mathbb R}(x^2+n^2)^{v-\frac12-s}
			e^{4it\arg(x+in)}\,\dd x\,\dd t\,\dd s.
			\tag*{(7.4)}
		\end{align*}
		On $\Re s=2$, all the integrals converge absolutely.
		
		When $n>0$, make the change of variables $x=n\cot\theta$. The standard beta integral gives
		\begin{align*}
			&\int_{\mathbb R}(x^2+n^2)^{v-\frac12-s}
			e^{4it\arg(x+in)}\,\dd x
			\\
			&\qquad=n^{2v-2s}\int_0^\pi
			(\sin\theta)^{2s-2v-1}e^{4it\theta}\,\dd\theta
			\\
			&\qquad=\frac{2\pi(2n)^{2v-2s}e(t)\Gamma(2s-2v)}
			{\Gamma(\frac12+2t-v+s)
				\Gamma(\frac12-2t-v+s)}.
			\tag*{(7.5)}
		\end{align*}
		When $n<0$, the phase $e(t)$ is replaced by $e(-t)$. This formula is first valid in the region of absolute convergence and is then used by analytic continuation.
		
		Add the terms corresponding to $n$ and $-n$, and use the gamma reflection formula
		\[
		\frac{1}{\Gamma(\frac12-2t-v+s)}
		=\frac{\Gamma(\frac12+2t+v-s)
			\cos(\pi(-2t-v+s))}{\pi}.
		\]
		Next use
		\[
		\cos(2\pi t)\cos(\pi(-2t-v+s))
		=\frac12\cos(\pi(v-s))
		+\frac12\cos(\pi(4t+v-s)).
		\]
		Define
		\[
		\mathcal I(v,s)=\int_K^{2K}\Phi\left(\frac{t}{K}\right)
		\frac{\Gamma(\frac12+2t+v+s)
			\Gamma(\frac12+2t+v-s)}
		{\Gamma(\frac12+2t)^2}\,\dd t.
		\tag*{(7.6)}
		\]
		Thus
		\begin{align*}
			\widehat F_n(0)+\widehat F_{-n}(0)
			={}&\frac{2}{2\pi i}\int_{(2)}H_u(s)\pi^{-2s}
			2^{2v-2s}\Gamma(2s-2v)\cos(\pi(v-s))
			\\
			&\qquad\times\mathcal I(v,s)n^{2v-2s}\,\dd s
			+\mathcal E_n.
			\tag*{(7.7)}
		\end{align*}
		where
		\begin{align*}
			\mathcal E_n
			={}&\frac{2}{2\pi i}\int_{(2)}H_u(s)\pi^{-2s}
			\Gamma(2s-2v)(2n)^{2v-2s}
			\\
			&\qquad\times\int_K^{2K}\Phi\left(\frac{t}{K}\right)
			\frac{\Gamma(\frac12+2t+v+s)
				\Gamma(\frac12+2t+v-s)}
			{\Gamma(\frac12+2t)^2}
			\\
			&\qquad\qquad\times
			\cos(\pi(4t+v-s))\,\dd t\,\dd s.
			\tag*{(7.8)}
		\end{align*}
		
		Integrate repeatedly by parts in the $t$-integral defining $\mathcal E_n$. Since $\Phi$ is smooth and supported on $[1,2]$, all boundary terms are $0$. Each differentiation gains a factor of $K^{-1+\epsilon}$. Shifting the $s$-contour to the right by any fixed distance then gives
		\[
		\mathcal E_n\ll_{\Phi,\epsilon,A}(Kn)^{-A},
		\qquad n\geq1.
		\tag*{(7.9)}
		\]
		After summing over $h,\eta$, the high-frequency cosine term in (7.9) contributes $O_{\Phi,\epsilon,A}(K^{-A})$. Here the exterior factor $L_0(1+2v)/\zeta(2+4v)$ contributes at most a fixed power of $\log K$, which can be absorbed by increasing $A$ in (7.9).
		
		Substitute (7.7) into (7.3). On $\Re s=2$, the series over $h$ and $\eta$ converge absolutely, and
		\begin{align*}
			\sum_{h\geq1}\sum_{\eta\geq1}
			h^{-1-4v}(h\eta)^{2v-2s}
			&=\left(\sum_{h\geq1}h^{-1-2v-2s}\right)
			\left(\sum_{\eta\geq1}\eta^{2v-2s}\right)
			\\
			&=\zeta(1+2v+2s)\zeta(2s-2v).
			\tag*{(7.10)}
		\end{align*}
		Therefore
		\begin{align*}
			(OD)_0
			={}&\frac{L_0(1+2v)}{\zeta(2+4v)}
			\frac{2}{2\pi i}\int_{(2)}H_u(s)\pi^{-2s}
			2^{2v-2s}\Gamma(2s-2v)\cos(\pi(v-s))
			\\
			&\qquad\times\mathcal I(v,s)
			\zeta(1+2v+2s)\zeta(2s-2v)\,\dd s
			+O_{\Phi,\epsilon,A}(K^{-A}).
			\tag*{(7.11)}
		\end{align*}
		
		The functional equation of the zeta function gives
		\[
		2\pi^{-2s}2^{2v-2s}\Gamma(2s-2v)
		\cos(\pi(v-s))\zeta(2s-2v)
		=\pi^{-2v}\zeta(1+2v-2s).
		\tag*{(7.12)}
		\]
		Hence
		\begin{align*}
			(OD)_0
			={}&\pi^{-2v}\frac{L_0(1+2v)}{\zeta(2+4v)}
			\frac{1}{2\pi i}\int_{(2)}H_u(s)\mathcal I(v,s)
			\\
			&\qquad\times
			\zeta(1+2v+2s)\zeta(1+2v-2s)\,\dd s
			+O_{\Phi,\epsilon,A}(K^{-A}).
			\tag*{(7.13)}
		\end{align*}
		
		The function $H_u(s)$ is odd, $\mathcal I(v,s)$ is even, and the product of the two zeta factors is also even. Thus the integrand in (7.13) is odd. Shift the contour from $\Re s=2$ to $\Re s=-2$. The integral over the left vertical line is the negative of the integral over the right vertical line, so the latter equals one half of the sum of the residues crossed. The poles crossed are at $s=\pm u,\pm v$. A direct calculation gives
		\begin{align*}
			&\frac{1}{2\pi i}\int_{(2)}H_u(s)\mathcal I(v,s)
			\zeta(1+2v+2s)\zeta(1+2v-2s)\,\dd s
			\\
			&\qquad=\mathcal I(v,u)
			\zeta(1+2v+2u)\zeta(1+2v-2u)
			\\
			&\qquad\quad-\frac12H_u(v)\mathcal I(v,v)\zeta(1+4v).
			\tag*{(7.14)}
		\end{align*}
		The uniform version of Stirling's formula gives
		\[
		\mathcal I(v,s)
		=2^{2v}K^{1+2v}\widetilde\Phi(2v)
		+O_{\Phi,\epsilon}\left(K^{2\Re v+\epsilon}\right),
		\qquad s\in\{u,v\}.
		\]
		The main integral above provides the meromorphic continuation in $v$. By (7.9), the $h,\eta$-series corresponding to the high-frequency term converges absolutely and locally uniformly in the target region, and is $O_{\Phi,\epsilon,A}(K^{-A})$. By the lower bounds on the shifts, the remaining factors in (7.14) cost at most a fixed power of $\log K$. Substituting (7.14) into (7.13) and then using the preceding formula shows that (7.1) holds uniformly throughout the target region.
	\end{proof}
	
	\section{Evaluation of $(OD)^*$}
	
	In this section we first establish a spectral identity in a right half-plane and then continue the spectral side to the target region. In the target region, the $q$-Dirichlet series generally does not converge, so below we first introduce truncations and then remove them by normal convergence.
	
	\subsection{Spectral expansion and a common transform estimate}
	
	For $h\geq 1$, $\eta\neq 0$, and $q\neq 0$, put
	\[
	g_{h,\eta,q}(C)
	=C^{-1-2v}\widehat F_{h\eta}\left(\frac{q}{hC}\right).
	\tag*{(8.1.1)}
	\]
	We first work in the half-plane $\Re v>1/2$. The continuation in $v$ will be made only after both sides of the spectral identity have been shown to converge normally.
	
	Fix $\omega\in C_c^\infty([0,\infty))$ with $\omega(x)=1$ for $0\leq x\leq1$. For $X\geq2$, put
	$g_{h,\eta,q}^{(X)}(C)=g_{h,\eta,q}(C)\omega(C/X)\omega(1/(XC))$.
	For positive integers $H,N,Q$, define
	\[
	\mathcal O_{H,N,Q,X}
	=\sum_{h\leq H}h^{-1-4v}
	\sum_{0<|\eta|\leq N}\sum_{0<|q|\leq Q}
	\sum_{C\geq1}R(q\eta;C)g_{h,\eta,q}^{(X)}(C).
	\]
	This is a finite linear combination of compactly supported Poincar\'e series. The spectral-expansion lemma in Section 2 therefore applies term by term and gives
	\begin{align*}
		\mathcal O_{H,N,Q,X}
		={}&\frac{1}{8\pi}
		\sum_{h\leq H}h^{-1-4v}
		\sum_{0<|\eta|\leq N}\sum_{0<|q|\leq Q}
		\int_{\mathbb R}
		\overline{\rho_t(q\eta)}E\left(i,\frac12+it\right)
		\mathcal K_{q\eta}g_{h,\eta,q}^{(X)}(t)\,\dd t
		\\
		&+\frac12
		\sum_{h\leq H}h^{-1-4v}
		\sum_{0<|\eta|\leq N}\sum_{0<|q|\leq Q}
		\sum_j\overline{\rho_j(q\eta)}u_j(i)
		\mathcal K_{q\eta}g_{h,\eta,q}^{(X)}(t_j).
	\end{align*}
	For fixed $h,\eta,q$ and $\Re v>0$, the derivatives of $g_{h,\eta,q}$ satisfy the hypotheses of the truncation lemma in Section 2. Indeed, the rapid decay of $\widehat F_{h\eta}$ gives arbitrary positive powers of $C$ as $C\to0$, whereas differentiation of $C^{-1-2v}\widehat F_{h\eta}(q/(hC))$ gives $O(C^{-1-2\Re v-r})$ after $r$ differentiations as $C\to\infty$. Hence, with $H,N,Q$ fixed, the limit $X\to\infty$ may be taken in the preceding finite identity. The passage $H,N,Q\to\infty$ will be justified after the common transform estimate below.
	
	For later reference, let $\Omega\Subset\{v:1/2<\Re v<1\}$ and choose
	$0<\epsilon_\Omega<2\inf_{v\in\Omega}\Re v-1$. We shall prove below that, for every sufficiently large $A$ and some $B=B(A,\Omega,c_0)>0$,
	\begin{align*}
		&\sum_{h\geq1}\sum_{\eta\neq0}h^{-1-4\Re v}
		\sum_{C\geq1}\sum_{q\neq0}
		\frac{|R(q\eta;C)|}{C^{1+2\Re v}}
		\left|\widehat F_{h\eta}\left(\frac{q}{hC}\right)\right|
		\\
		&\qquad\ll_{A,\Omega,\Phi,c_0}K^B
		\sum_{h\geq1}h^{-4\Re v}
		\sum_{\eta\neq0}(1+h|\eta|)^{-A}
		\sum_{C\geq1}C^{-2\Re v+\epsilon_\Omega}<\infty,
		\tag*{(8.1.2)}
	\end{align*}
	locally uniformly for $v\in\Omega$. We shall also prove that the limit of the spectral side is absolutely convergent and equals
	\begin{align*}
		(OD)^*&=(OD)_E+(OD)_M,
		\\
		(OD)_E
		&=\frac{1}{8\pi}
		\sum_{h\geq1}h^{-1-4v}
		\sum_{\eta\neq0}\sum_{q\neq0}
		\int_{\mathbb R}
		\overline{\rho_t(q\eta)}E\left(i,\frac12+it\right)
		\mathcal K_{q\eta}g_{h,\eta,q}(t)\,\dd t,
		\\
		(OD)_M
		&=\frac12
		\sum_{h\geq1}h^{-1-4v}
		\sum_{\eta\neq0}\sum_{q\neq0}
		\sum_j\overline{\rho_j(q\eta)}u_j(i)
		\mathcal K_{q\eta}g_{h,\eta,q}(t_j).
		\tag*{(8.1.3)}
	\end{align*}
	Until the convergence proof at the end of the subsection, the right-hand side of (8.1.3) is only shorthand for the corresponding iterated truncation: first $X\to\infty$ with $H,N,Q$ fixed, and then rectangular limits in $H,N,Q$. Its existence as an ordinary spectral sum is not being assumed here.
	
	To put the Bessel transforms into a form independent of $q$, use $\kappa$ for the real variable in the definition of $W_1$. For $n\neq0$, set
	\[
	\mathscr B_n^\pm(t)
	=\int_0^\infty\widehat F_n(\pm y)y^{-\frac12+2v}
	K_{it}(2\pi|n|y)\,\dd y,
	\qquad
	\mathscr B_n(t)=\mathscr B_n^+(t)+\mathscr B_n^-(t).
	\]
	The change of variables $y=|q|/(hC)$ gives the exact identity
	\[
	\mathcal K_{q\eta}g_{h,\eta,q}(t)
	=h^{\frac12+2v}|q|^{-\frac12-2v}
	\mathscr B_{h\eta}^{\operatorname{sgn}(q)}(t).
	\tag*{(8.1.4)}
	\]
	
	\begin{lemma}[Uniform transform estimate]
		Fix $A>0$, integers $j\geq0$ and $r\geq0$, and $\epsilon>0$. Suppose that
		$|\Re u|+|\Re v|\leq c_0/\log K$ and $|u|,|v|<\log K$, where $c_0$ is fixed. Then, uniformly for $n\neq0$, $x\in\mathbb R$, and $t\in\mathbb R$,
		\begin{align*}
			F_n^{(j)}(x)
			&\ll_{A,j,\Phi,\epsilon,c_0}
			K^{-j+\epsilon}
			\left(1+\frac{|x|}{K}\right)^{-A}(1+|n|)^{-A},
			\\
			e^{\frac\pi2|t|}
			\left(|\mathscr B_n^+(t)|+|\mathscr B_n^-(t)|\right)
			&\ll_{A,\Phi,\epsilon,c_0}
			K^{\frac12+\epsilon}(1+|t|)^{-A}(1+|n|)^{-A}.
			\tag*{(8.1.5)}
		\end{align*}
		The same estimates remain valid with $F_n^{(j)}$ and $\mathscr B_n^\pm$ replaced by
		$\partial_u^a\partial_v^bF_n^{(j)}$ and
		$\partial_u^a\partial_v^b\mathscr B_n^\pm$, respectively, whenever $a+b\leq r$, at the cost of replacing $\epsilon$ by $2\epsilon$.
		
		Let $0\leq\sigma_0<1/2$ be fixed. The functions $\mathscr B_n^\pm(t)$ are holomorphic for
		$|\Im t|<\Re(1/2+2v)$. Uniformly for $|\Im t|\leq\sigma_0$ and sufficiently large $K$,
		\[
		e^{\frac\pi2|\Re t|}
		\left(|\mathscr B_n^+(t)|+|\mathscr B_n^-(t)|\right)
		\ll_{A,\sigma_0,\Phi,\epsilon,c_0}
		K^{\frac12+\sigma_0+\epsilon}
		(1+|t|)^{-A}(1+|n|)^{-A}.
		\tag*{(8.1.6)}
		\]
		
		There is also a fixed-strip form. Let
		$\Omega\Subset\{v:-1/4<\Re v<1\}$, put
		$\sigma_\Omega=\inf_{v\in\Omega}\Re(1/2+2v)>0$, and fix
		$0\leq\sigma_0<\min(\sigma_\Omega,1/2)$. There is a number
		$B=B(A,j,r,\Omega,\sigma_0,c_0)>0$ such that, uniformly for
		$v\in\Omega$, $|\Re u|\leq c_0/\log K$, $|u|<\log K$,
		$a+b\leq r$, and $|\Im t|\leq\sigma_0$,
		\begin{align*}
			|\partial_u^a\partial_v^bF_n^{(j)}(x)|
			&\ll_{A,j,r,\Omega,\Phi,c_0}K^{B-j}
			\left(1+\frac{|x|}{K}\right)^{-A}(1+|n|)^{-A},
			\\
			e^{\frac\pi2|\Re t|}
			\left(|\partial_u^a\partial_v^b\mathscr B_n^+(t)|
			+|\partial_u^a\partial_v^b\mathscr B_n^-(t)|\right)
			&\ll_{A,r,\Omega,\sigma_0,\Phi,c_0}K^B
			(1+|t|)^{-A}(1+|n|)^{-A}.
			\tag*{(8.1.7)}
		\end{align*}
		In particular, (8.1.7) is a statement with fixed quantifiers; it is the crude-strip estimate, including the compact-complex-$t$ variant, used in the continuation arguments below.
	\end{lemma}
	
	\begin{proof}
		Put $Z=2\kappa+1/2$ and
		$G_\kappa(v,s)=\Gamma(Z+v+s)\Gamma(Z-v+s)/\Gamma(Z)^2$. Also put
		$B_\kappa(R)=R^{2v}V_\kappa(R^2)$ for $R\geq1$. Then
		\[
		B_\kappa(R)
		=\frac{R^{-1+2v}}{2\pi i}
		\int_{(1)}H_u(s)G_\kappa(v,s)(\pi R)^{-2s}\,\dd s.
		\tag*{(8.1.8)}
		\]
		We begin with a uniform contour estimate. Fix $L>1$ and integers $a,m\geq0$. For
		$c\in\{-L,L\}$, repeated differentiation of the uniform Stirling expansion gives
		\[
		\int_{\mathbb R}|H_u(c+i\tau)|
		(1+|c+i\tau|+|v|)^m
		|\partial_\kappa^aG_\kappa(v,c+i\tau)|\,\dd\tau
		\ll_{a,m,L,\epsilon}K^{2c-a+\epsilon}.
		\tag*{(8.1.9)}
		\]
		This holds for every fixed $a,m,L$ under the stated shift bounds. To verify (8.1.9), split the $\tau$-line into the two sets
		$|\tau-\Im u|\leq(\log(2K))^2$ and
		$|\tau+\Im u|\leq(\log(2K))^2$, and their complement. On the first two sets, the Gamma arguments have real parts comparable with $K$ and imaginary parts $O((\log(2K))^2)=o(K)$. Stirling's formula is therefore uniform. Each $\kappa$-derivative contributes $K^{-1}$, together with a fixed polynomial in $1+|u|+|v|+|\tau|$, whose contribution is absorbed into $K^\epsilon$. On the complement, the real parts of the Gamma arguments are still positive and comparable with $K$. The inequality $|\Gamma(x+iy)|\leq\Gamma(x)$ for $x>0$, together with the integral representations for the differentiated Gamma ratios, leaves at most a fixed polynomial in $K+|\tau|$. The Gaussian factors $e^{-(\tau-\Im u)^2}$ and $e^{-(\tau+\Im u)^2}$ in the two summands of $H_u$ dominate this polynomial. This proves the integral estimate. The same argument at the residue points gives
		\[
		|\partial_\kappa^aG_\kappa(v,\xi)|
		\ll_{a,L,\epsilon}K^{2\Re\xi-a+\epsilon},
		\qquad \xi\in\{u,-u\}.
		\tag*{(8.1.10)}
		\]
		
		The lines $\Re s=\pm L$ stay a fixed distance from the poles $s=\pm u$, and the Gaussian bounds make the horizontal integrals tend to zero. No Gamma pole is crossed when the line in (8.1.8) is moved to $\Re s=\pm L$, since $Z\asymp K$ and $L$ is fixed. If $1\leq R\leq K$, move the line to $\Re s=-L$ and collect the two poles of $H_u$, each of residue $1$. If $R>K$, move it to $\Re s=L$; in this direction no pole is crossed. Thus
		\begin{align*}
			B_\kappa(R)
			={}&\sum_{\xi=u,-u}
			\pi^{-2\xi}G_\kappa(v,\xi)R^{-1+2v-2\xi}
			\\
			&+\frac{R^{-1+2v}}{2\pi i}
			\int_{(-L)}H_u(s)G_\kappa(v,s)(\pi R)^{-2s}\,\dd s,
			\qquad 1\leq R\leq K,
			\tag*{(8.1.11)}
		\end{align*}
		and
		\[
		B_\kappa(R)
		=\frac{R^{-1+2v}}{2\pi i}
		\int_{(L)}H_u(s)G_\kappa(v,s)(\pi R)^{-2s}\,\dd s,
		\qquad R>K.
		\tag*{(8.1.12)}
		\]
		Here the notation $\sum_{\xi=u,-u}$ counts the two entries with multiplicity. Thus the two residue terms in (8.1.11) are kept together when $u$ approaches zero; no division by $u$ is used anywhere.
		
		Let $R=(x^2+n^2)^{1/2}$. Differentiating (8.1.11) and (8.1.12), and then applying (8.1.9)--(8.1.10), gives, for fixed $a,b,M\geq0$,
		\[
		|\partial_\kappa^a\partial_x^bB_\kappa(R)|
		\ll_{a,b,M,\Phi,\epsilon}
		K^{-a+\epsilon}R^{-1-b}
		\left(1+\frac{R}{K}\right)^{-M}.
		\tag*{(8.1.13)}
		\]
		For example, on $\Re s=-L$ the factor left after (8.1.9) is
		$K^{-2L}R^{2L+2\Re v}$, which is bounded for $R\leq K$; on $\Re s=L$ it is
		$K^{2L}R^{-2L+2\Re v}$, which is $O((R/K)^{-M})$ once $L$ is chosen in terms of $M$. The residue terms satisfy the same estimate because
		$K^{2\Re\xi}R^{2\Re v-2\Re\xi}=O(1)$ for $R\leq K$ and $\xi=\pm u$. Since $|u|,|v|<\log K$, all fixed powers of the shifts are absorbed into $K^\epsilon$. Differentiation with respect to $u$ or $v$ introduces only fixed powers of logarithms and of the shifts, so the same argument proves the derivative assertion in the lemma.
		
		Write $\vartheta=\arg(x+in)$, with the branch fixed in Section 4. Since $n\neq0$, this is smooth as a function of $x$. For $a\geq1$,
		$|\vartheta^{(a)}(x)|\ll_a R^{-a}\min(1,|\vartheta|)$, and
		$|n|\leq R\min(1,|\vartheta|)$. With
		$A_\kappa(x)=\Phi(\kappa/K)B_\kappa(R)$, (8.1.13) yields
		\[
		|\partial_\kappa^p\partial_x^qA_\kappa(x)|
		\ll_{p,q,M,\Phi,\epsilon}
		K^{-p+\epsilon}R^{-1-q}
		\left(1+\frac{R}{K}\right)^{-M}.
		\tag*{(8.1.14)}
		\]
		
		We now estimate
		$F_n(x)=\int_K^{2K}e^{4i\kappa\vartheta}A_\kappa(x)\,\dd\kappa$.
		If $K|\vartheta|\leq1$, then $R\geq K|n|$. Direct differentiation under the integral, followed by (8.1.14), gives
		\[
		|F_n^{(j)}(x)|
		\ll_{A,j,\Phi,\epsilon}K^{-j+\epsilon}
		\left(1+\frac{|x|}{K}\right)^{-A}(1+|n|)^{-A}.
		\tag*{(8.1.15)}
		\]
		If $K|\vartheta|>1$, use
		$e^{4i\kappa\vartheta}=(4i\vartheta)^{-1}\partial_\kappa e^{4i\kappa\vartheta}$.
		After differentiating $j$ times with respect to $x$, integrate by parts in $\kappa$ exactly $N+2j+2$ times. Every boundary term vanishes because $\Phi$ and all its derivatives vanish at the endpoints of its support. An induction using (8.1.14) gives
		\[
		|F_n^{(j)}(x)|
		\ll_{j,N,M,\Phi,\epsilon}
		K^{-j+\epsilon}
		\left(1+\frac{R}{K}\right)^{-M}
		(1+K|\vartheta|)^{-N}.
		\tag*{(8.1.16)}
		\]
		Since
		$(1+R/K)(1+K|\vartheta|)\geq1+|x|/K+|n|$, choosing $M$ and $N$ in terms of $A$ proves the first estimate in (8.1.5).
		
		For the Bessel transform, let $f_n(X)=F_n(KX)$. The estimate just proved implies, for fixed $b,L,M\geq0$,
		\[
		\int_{\mathbb R}(1+|X|)^L|f_n^{(b)}(X)|\,\dd X
		\ll_{b,L,M,\Phi,\epsilon}K^\epsilon(1+|n|)^{-M}.
		\tag*{(8.1.17)}
		\]
		Differentiation under the Fourier integral and integration by parts therefore give
		\[
		(1+|Y|)^b|\partial_Y^\ell\widehat f_n(Y)|
		\ll_{b,\ell,M,\Phi,\epsilon}K^\epsilon(1+|n|)^{-M}.
		\tag*{(8.1.18)}
		\]
		Put $a_n=2\pi|n|/K$, $\beta=1/2+2v$, and
		$A_n^\pm(Y)=Y^\beta\widehat f_n(\pm Y)$. Since
		$\widehat F_n(y)=K\widehat f_n(Ky)$, a change of variables gives
		\[
		\mathscr B_n^\pm(t)
		=K^{\frac12-2v}\int_0^\infty
		A_n^\pm(Y)K_{it}(a_nY)\,\frac{\dd Y}{Y}.
		\tag*{(8.1.19)}
		\]
		Let $D=Y\partial_Y$ and $\mathcal L_n=D^2-(a_nY)^2$. From (8.1.18), for fixed $m,B,M$,
		\[
		|D^mA_n^\pm(Y)|
		\ll_{m,B,M,\Phi,\epsilon}K^\epsilon(1+|n|)^{-M}
		\begin{cases}
			Y^{\Re\beta},&0<Y\leq1,\\
			Y^{-B},&Y\geq1.
		\end{cases}
		\tag*{(8.1.20)}
		\]
		The Bessel equation is $\mathcal L_nK_{it}(a_nY)=-t^2K_{it}(a_nY)$. Green's identity for $\mathcal L_n$ with measure $\dd Y/Y$ has boundary term
		$[(DP)Q-P(DQ)]_0^\infty$. It vanishes for
		$P=(1-\mathcal L_n)^mA_n^\pm$ and $Q=K_{it}(a_nY)$: at zero this follows from $\Re\beta>0$ and the small-argument expansion of $K_{it}$, and at infinity from rapid decay. Thus, for every integer $R_0\geq0$,
		\[
		\int_0^\infty A_n^\pm(Y)K_{it}(a_nY)\,\frac{\dd Y}{Y}
		=(1+t^2)^{-R_0}
		\int_0^\infty(1-\mathcal L_n)^{R_0}A_n^\pm(Y)
		K_{it}(a_nY)\,\frac{\dd Y}{Y}.
		\tag*{(8.1.21)}
		\]
		
		Set $d_K=1/\log(2K)$. Expanding $(1-\mathcal L_n)^{R_0}$ and using (8.1.20), with a larger initial power of $(1+|n|)^{-1}$, gives
		\[
		\int_0^\infty
		|(1-\mathcal L_n)^{R_0}A_n^\pm(Y)|^2Y^{-2d_K}\,\frac{\dd Y}{Y}
		\ll_{R_0,M,\Phi,\epsilon}K^\epsilon(1+|n|)^{-2M}.
		\tag*{(8.1.22)}
		\]
		For $d>0$, the classical product integral is
		\[
		\int_0^\infty|K_{it}(x)|^2x^{2d}\,\frac{\dd x}{x}
		=\frac{2^{2d-3}\Gamma(d)^2\Gamma(d+it)\Gamma(d-it)}{\Gamma(2d)}.
		\tag*{(8.1.23)}
		\]
		With $d=d_K$, Stirling's formula, including the bounded $t$-range, yields
		\[
		\left(\int_0^\infty|K_{it}(a_nY)|^2Y^{2d_K}\,\frac{\dd Y}{Y}\right)^{1/2}
		\ll(\log(2K))^2a_n^{-d_K}
		e^{-\frac\pi2|t|}(1+|t|)^{d_K-\frac12}.
		\tag*{(8.1.24)}
		\]
		Since $n\neq0$, one has $a_n^{-d_K}\ll1$. Apply Cauchy--Schwarz to (8.1.21), then use (8.1.22), (8.1.24), and the preceding estimate. Because $|K^{1/2-2v}|\ll K^{1/2}$ in the central strip, we obtain
		\[
		e^{\frac\pi2|t|}|\mathscr B_n^\pm(t)|
		\ll_{R_0,M,\Phi,\epsilon}
		K^{\frac12+\epsilon}(1+t^2)^{-R_0}
		(1+|t|)^{d_K-\frac12}(1+|n|)^{-M}.
		\]
		Taking $R_0$ and $M$ sufficiently large proves the second estimate in (8.1.5).
		
		Now let $t=T+i\sigma$ with $|\sigma|\leq\sigma_0<1/2$. The small-argument expansion of $K_{it}$ and (8.1.20) show that (8.1.19) is locally uniformly convergent, hence holomorphic, whenever $|\Im t|<\Re\beta$. Take
		$d=\sigma_0+d_K$; for sufficiently large $K$, one has $d<\Re\beta$. The corresponding product formula is
		\[
		\int_0^\infty|K_{it}(x)|^2x^{2d}\,\frac{\dd x}{x}
		=\frac{2^{2d-3}}{\Gamma(2d)}
		\Gamma(d-\sigma)\Gamma(d+\sigma)
		\Gamma(d+iT)\Gamma(d-iT).
		\tag*{(8.1.25)}
		\]
		Uniform Stirling estimates give the same exponential decay as before, while
		$a_n^{-d}\ll_{\sigma_0}K^{\sigma_0+\epsilon}$. Repeating the argument based on (8.1.21) proves (8.1.6).
		
		Finally, let $\Omega\Subset\{-1/4<\Re v<1\}$ and retain the notation $\sigma_\Omega$ from the statement. On $\Omega$, the contour argument remains valid because $\Re\beta\geq\sigma_\Omega>0$. The residue terms in (8.1.11) and the prefactor in (8.1.19) may now contribute a fixed power of $K$, depending on $\Omega$. For $|\Im t|\leq\sigma_0<\min(\sigma_\Omega,1/2)$, take
		$d=(\sigma_0+\sigma_\Omega)/2$ in the product integral; then $|\Im t|<d<\Re\beta$ uniformly on $\Omega$, so both the Green identity and the weighted $L^2$ argument are uniform. Keeping all fixed powers of $K$ instead of absorbing them into $K^\epsilon$ proves (8.1.7). The same proof applies after a fixed number of $u$- and $v$-derivatives, because these derivatives introduce only powers of logarithms and polygamma functions. This completes the proof.
	\end{proof}
	
	We now remove the truncations. Let
	$\Omega\Subset\{v:1/2<\Re v<1\}$. The Fourier-transform formula, the fixed-strip estimate for $F_n^{(j)}$, and sufficiently many integrations by parts give, after enlarging $B$ if necessary,
	\[
	|\widehat F_n(y)|
	\ll_{A,\Omega,\Phi,c_0}K^{B+1}(1+K|y|)^{-A}(1+|n|)^{-A}.
	\]
	For $a>0$ and $A>1$, comparison with an integral gives
	$\sum_{q\neq0}(1+a|q|)^{-A}\ll_A\min(a^{-1},a^{-A})$.
	Taking $a=K/(hC)$ therefore yields
	\[
	\sum_{q\neq0}
	\left|\widehat F_{h\eta}\left(\frac{q}{hC}\right)\right|
	\ll_{A,\Omega,\Phi,c_0}K^B hC(1+h|\eta|)^{-A}.
	\]
	Since $|R(m;C)|\leq R(0;C)\ll_{\epsilon_\Omega}C^{\epsilon_\Omega}$, substitution of this estimate proves (8.1.2). Notice that the choice of $\epsilon_\Omega$ made above is exactly what is needed for uniform convergence of the $C$-series. Thus the geometric side is absolutely and locally uniformly convergent on $\Omega$.
	
	It remains to prove the corresponding assertion on the spectral side. Only even cusp forms contribute at $i$, since $u_j(i)=0$ for an odd form; for an even form write $\lambda_j(m)=\rho_j(m)/\rho_j(1)$ for $m\geq1$. For real $t$, the standard Fourier-coefficient formulas and Stirling's formula give a fixed $B_0>0$ such that
	$|E(i,1/2+it)|\ll(1+|t|)^{B_0}$ and
	$e^{-\pi|t|/2}|\xi(1+2it)|^{-1}\ll(1+|t|)^{B_0}$; at $t=0$ the latter expression is interpreted by continuity. The bounds
	$|\sigma_{2it}(m)|\leq d(m)$ and
	$|\lambda_j(m)|\leq d(m)m^{1/2}$ imply, uniformly for $v\in\Omega$,
	\begin{align*}
		\sum_{q\geq1}\frac{|\sigma_{2it}(q|\eta|)|}{q^{\frac12+2\Re v}}
		&\ll_{\Omega,\epsilon_\Omega}|\eta|^{\epsilon_\Omega},
		\\
		\sum_{q\geq1}\frac{|\lambda_j(q|\eta|)|}{q^{\frac12+2\Re v}}
		&\ll_{\Omega,\epsilon_\Omega}|\eta|^{\frac12+\epsilon_\Omega}.
	\end{align*}
	After (8.1.4) is inserted, the remaining arithmetic majorants are bounded by
	\[
	\sum_{h\geq1}h^{-\frac12-2\Re v}
	\sum_{\eta\neq0}(1+h|\eta|)^{-A}|\eta|^\theta,
	\qquad
	\theta\in\{\epsilon_\Omega,1/2+\epsilon_\Omega\},
	\]
	and these converge uniformly on $\Omega$ once $A$ is large enough. In the Eisenstein term, (8.1.7) cancels the exponential growth of $|\xi(1+2it)|^{-1}$ and leaves an arbitrarily large negative power of $1+|t|$; hence the $t$-integral is absolutely convergent.
	
	For the Maa\ss{} spectrum, use
	$e^{-\pi t_j/2}|\rho_j(1)|\ll(1+t_j)^{B_0}$ and the local Weyl estimate
	$\sum_{T\leq t_j<2T}|u_j(i)|\ll(1+T)^2$. The latter follows from the pointwise local Weyl law and Cauchy--Schwarz. On the dyadic block $T\leq t_j<2T$, (8.1.7) gives a bound $O((1+T)^{-A+B_0+2})$ after the arithmetic sums have been taken. The sum over dyadic $T$ therefore converges when $A>B_0+3$. Any exceptional spectral parameters form a finite set in $|\Im t_j|<1/2$ and are covered by the compact-complex-$t$ part of (8.1.7). Every summand and integrand is holomorphic in $v$, and all the majorants just obtained are uniform on compact subsets of $\Omega$. The Eisenstein and Maa\ss{} expressions consequently converge normally there.
	
	We may now take $X\to\infty$ first with $H,N,Q$ fixed, as prescribed by the truncation lemma, and then let $H,N,Q\to\infty$. Absolute convergence makes the order of the latter three limits immaterial and identifies the result with the ordinary sums and integral in (8.1.3). Thus (8.1.3) is proved for $1/2<\Re v<1$, where it agrees with the absolutely convergent geometric expression (6.14).
	
	The fixed-strip estimate (8.1.7) supplies the uniform transform majorant needed after the $q$-series is rewritten in the next two subsections. The normal convergence of those rewritten expressions is checked there on their respective domains; this is what justifies the later analytic continuations of the two spectral components.
	
	\subsection{Eisenstein spectrum}
	
	By the formula for the Eisenstein Fourier coefficients, for $\Re v>1/4$ we have, with absolute convergence,
	\begin{align*}
		\mathscr D_\eta(t;v)
		&=\sum_{q\geq1}\frac{\sigma_{2it}(q|\eta|)}{q^{\frac12+2v+it}}
		=\zeta\left(\frac12+2v+it\right)
		\zeta\left(\frac12+2v-it\right)\mathscr P_\eta(t;v),
		\\
		\mathscr P_\eta(t;v)
		&=\prod_{p^a\parallel|\eta|}
		\left(
		\sigma_{2it}(p^a)-p^{-\frac12-2v+it}\sigma_{2it}(p^{a-1})
		\right).
		\tag*{(8.2.1)}
	\end{align*}
	The Euler identity in (8.2.1) follows first in an absolute half-plane by checking one prime at a time, and then throughout $\Re v>1/4$ by absolute convergence of the $q$-series. Substituting (8.1.4) into (8.1.3) and combining the terms with $q>0$ and $q<0$, we obtain
	\begin{align*}
		(OD)_E
		={}&\frac{1}{4\pi}
		\sum_{h\geq1}h^{-\frac12-2v}
		\sum_{\eta\neq0}
		\int_{\mathbb R}
		\frac{E(i,\frac12+it)}{\xi(1-2it)}
		|\eta|^{-it}
		\mathscr D_\eta(t;v)
		\mathscr B_{h\eta}(t)\,\dd t.
		\tag*{(8.2.2)}
	\end{align*}
	In every closed strip contained in $1/4<\Re v<1$, the crude-strip form of (8.1.7), the elementary estimate
	\[
	\mathscr P_\eta(t;v)\ll_{\epsilon}|\eta|^\epsilon,
	\qquad t\in\mathbb R,
	\]
	and polynomial vertical-strip bounds for the zeta and Eisenstein factors show that the sum of the integrals of the absolute values in (8.2.2) is finite. Thus all the interchanges used in deriving (8.2.2) are justified there.
	
	We next continue (8.2.2) across $\Re v=1/4$. Put
	\[
	\mathscr A(t)=\frac{E(i,\frac12+it)}{\xi(1-2it)}.
	\]
	The function $\mathscr A$ is meromorphic in $t$. It has no pole on the real axis. Indeed, for $t\neq0$ this follows from the zero-free line $\zeta(1+it)\neq0$, and at $t=0$ it follows from
	\[
	E\left(i,\frac12+it\right)
	=\frac{2\zeta(\frac12+it)L(\frac12+it,\chi_{-4})}{\zeta(1+2it)};
	\]
	the apparent singularity of $\mathscr A$ at $0$ is removable. The
	scattering poles of $\mathscr A$ form a discrete set off the real axis.
	Choose $0<\rho<1/16$ so small that $|t|\leq4\rho$ contains no such pole.
	
	For later use, we record the normal-convergence assertion needed on the local contour. If $U$ is a sufficiently small neighborhood of $v=1/4$ and $|t|\leq4\rho$, then, for every $A>0$,
	\begin{align*}
		\mathscr B_{h\eta}(t)&\ll_{A,U,\Phi}K^B(1+h|\eta|)^{-A},
		\\
		\mathscr P_\eta(t;v)&\ll_{U,\rho,\epsilon}|\eta|^{16\rho+\epsilon}.
	\end{align*}
	Here the first estimate is the compact-complex-$t$ version of the crude strip estimate (8.1.7). It follows directly from (8.1.19), the small-argument expansion of $K_{it}$, and repeated integration by parts in the Fourier transform. The second estimate follows prime by prime from (8.2.1). Consequently, after increasing $A$, the $h,\eta$-series of the integrand converges normally on every local contour used below. More explicitly, on such a contour it is dominated, apart from a constant depending on the contour, by
	\[
	K^B\sum_{h\geq1}h^{-\frac12-2\Re v}
	\sum_{\eta\neq0}(1+h|\eta|)^{-A}|\eta|^{16\rho+\epsilon},
	\]
	which is finite when $A$ is sufficiently large, uniformly for $v$ in a
	compact subset of $U$ and for contours staying a fixed distance from the
	moving poles. Along the displaced path $\Im v=\tau$ used below, the two
	moving poles remain distinct. In disjoint tubular neighborhoods of their
	trajectories, the same majorant holds after multiplication by
	$t-t_+(v)$ or $t-t_-(v)$, respectively. No assertion at the collision
	point $(v,t)=(1/4,0)$ is needed, since this path avoids that point. Hence
	the two simple residues may be taken term by term and then summed.
	
	The only poles depending on $v$ that cross the real $t$-axis are
	\[
	t_+(v)=i\left(2v-\frac12\right),
	\qquad
	t_-(v)=-i\left(2v-\frac12\right).
	\tag*{(8.2.3)}
	\]
	To fix the signs unambiguously, start with real $v>1/4$ close to $1/4$, move first to $v+i\tau$ with $0<\tau<\rho/8$, decrease $\Re v$ through $1/4$, and then return to the real $v$-axis. At the crossing, $t_+$ and $t_-$ meet the real axis at $-2\tau$ and $2\tau$, respectively. Choose disjoint discs $D_+$ and $D_-$ centered at these two points, with closures contained in $|t|<\rho$. For a truncated integral over $[-T,T]$, split the contour into its portions inside $D_+$, inside $D_-$, and the complementary portion. Only the first portion is deformed when $t_+$ crosses, and only the second is deformed when $t_-$ crosses; the complementary portion is kept fixed. Equivalently, in each disc replace the real diameter by a small detour with the same endpoints. This local splitting avoids any ambiguity caused by the interchange of the vertical positions of $t_+$ and $t_-$. It also keeps the contour equal to the real axis outside $D_+\cup D_-$, so no scattering pole is crossed. Letting $T\to\infty$ is legitimate and uniform by the preceding majorant.
	
	After the deformed contour is pulled back to the real axis, the downward crossing of $t_+$ contributes $2\pi i$ times its residue, whereas the upward crossing of $t_-$ contributes $-2\pi i$ times its residue. Hence, initially for real $0<v<1/4$,
	\[
	\mathscr G_{h,\eta}(t;v)
	=\mathscr A(t)|\eta|^{-it}
	\zeta\left(\frac12+2v+it\right)
	\zeta\left(\frac12+2v-it\right)
	\mathscr P_\eta(t;v)\mathscr B_{h\eta}(t),
	\]
	and
	\begin{align*}
		(OD)_E
		={}&\mathcal E_{\mathrm{reg}}(v)
		+\frac{1}{4\pi}\sum_{h\geq1}h^{-\frac12-2v}
		\sum_{\eta\neq0}
		\bigg\{
		2\pi i\operatorname*{Res}_{t=t_+(v)}\mathscr G_{h,\eta}(t;v)
		-2\pi i\operatorname*{Res}_{t=t_-(v)}\mathscr G_{h,\eta}(t;v)
		\bigg\},
	\end{align*}
	where
	\begin{align*}
		\mathcal E_{\mathrm{reg}}(v)
		={}&\frac{1}{4\pi}
		\sum_{h\geq1}h^{-\frac12-2v}
		\sum_{\eta\neq0}
		\int_{\mathbb R}
		\mathscr A(t)|\eta|^{-it}
		\\
		&\qquad\times
		\zeta\left(\frac12+2v+it\right)
		\zeta\left(\frac12+2v-it\right)
		\mathscr P_\eta(t;v)\mathscr B_{h\eta}(t)\,\dd t.
	\end{align*}
	The residues of the first and second zeta factors at $t_+$ and $t_-$ are $-i$ and $i$, respectively. It follows, with every sum normally convergent, that the two crossing contributions are
	\begin{align*}
		\mathcal R_+(v)
		={}&\frac12\frac{E(i,1-2v)}{\xi(4v)}\zeta(4v)
		\sum_{h\geq1}h^{-\frac12-2v}
		\sum_{\eta\neq0}|\eta|^{2v-\frac12}
		\mathscr P_\eta(t_+;v)\mathscr B_{h\eta}(t_+),
		\\
		\mathcal R_-(v)
		={}&\frac12\frac{E(i,2v)}{\xi(2-4v)}\zeta(4v)
		\sum_{h\geq1}h^{-\frac12-2v}
		\sum_{\eta\neq0}|\eta|^{\frac12-2v}
		\mathscr P_\eta(t_-;v)\mathscr B_{h\eta}(t_-).
	\end{align*}
	The signs in both formulas are positive: for $t_+$ the factor $2\pi i$ is multiplied by the residue $-i$, while for $t_-$ the factor $-2\pi i$ is multiplied by the residue $i$.
	
	The identities
	\begin{align*}
		\mathscr B_n(t_-)&=\mathscr B_n(t_+),
		\\
		\mathscr P_\eta(t_-;v)&=|\eta|^{4v-1}\mathscr P_\eta(t_+;v),
		\\
		\frac{E(i,2v)}{\xi(2-4v)}&=\frac{E(i,1-2v)}{\xi(4v)}
	\end{align*}
	show that $\mathcal R_-(v)=\mathcal R_+(v)$. The first identity uses $K_{it}=K_{-it}$, the second follows from the finite Euler product, and the third follows from the functional equations of the Eisenstein series and $\xi$. We have therefore proved
	\[
	(OD)_E=\mathcal E_{\mathrm{reg}}(v)+\mathcal R(v),
	\qquad
	\mathcal R(v)=2\mathcal R_+(v),
	\]
	initially for $0<v<1/4$. Both sides are meromorphic in $v$, so the identity theorem extends this equality to complex $v$ wherever the expressions below are defined.
	
	We now evaluate $\mathcal R(v)$. We use the cosine--Bessel kernel identity
	\[
	\int_0^\infty \cos(by)y^\nu K_\nu(ay)\,\dd y
	=\frac{\sqrt\pi(2a)^\nu\Gamma(\nu+\frac12)}
	{2(a^2+b^2)^{\nu+\frac12}},
	\qquad \Re\nu>-\frac12,\quad a>0.
	\]
	For $\Re v>1/4$, take $\nu=2v-\frac12$, $a=2\pi|n|$, and $b=2\pi x$. Since $K_{-\nu}=K_\nu$, Fourier inversion gives
	\[
	\widehat F_n(y)+\widehat F_n(-y)
	=2\int_{\mathbb R}F_n(x)\cos(2\pi xy)\,\dd x.
	\]
	Substituting this identity into the definition of $\mathscr B_n(t_+(v))$, applying Fubini's theorem, and using $F_n(x)=(x^2+n^2)^vW_1(x+in)$ gives
	\[
	\mathscr B_n(t_+(v))
	=\frac{\Gamma(2v)}{2\pi^{2v}}|n|^{2v-\frac12}
	\int_{\mathbb R}(x^2+n^2)^{-v}W_1(x+in)\,\dd x.
	\tag*{(8.2.4)}
	\]
	The integral defining the left-hand side at $t=t_+(v)$ and the integral
	on the right are both holomorphic in the connected strip
	$0<\Re v<1$, locally uniformly in the stated shift range; this follows
	from the integral representations and the crude-strip estimate (8.1.7).
	Thus (8.2.4) holds throughout $0<\Re v<1$ by the identity
	theorem.
	
	Rearranging the normally convergent $h,\eta$-series by $n=h\eta$, using the prime-power identity
	\[
	\sum_{h\mid n}h^{-4v}
	\mathscr P_{n/h}(t_+(v);v)=\sigma_{1-4v}(n),
	\qquad n\geq1,
	\]
	and then applying $n^{4v-1}\sigma_{1-4v}(n)=\sigma_{4v-1}(n)$, we obtain
	\[
	\mathcal R(v)
	=\frac{L_0(1-2v)}{\zeta(2-4v)}
	\sum_{n\neq0}\sigma_{4v-1}(|n|)
	\int_{\mathbb R}(x^2+n^2)^{-v}W_1(x+in)\,\dd x.
	\]
	All rearrangements here are justified first for $0<\Re v<1/4$ by the rapid decay in $n$ inherited from (8.1.5). Equivalently, one may truncate $h,\eta$ and then pass to the limit by normal convergence. This also justifies interchanging the residue operation with the $h,\eta$-sums.
	
	To continue the last expression across $\Re v=0$, truncate the $n$-sum at $|n|\leq N$, insert the Mellin representation of $W_1$, and pair $n$ with $-n$. For fixed $N$ all interchanges are finite. The beta integral, the gamma reflection formula, and the functional equation of the zeta function show that the nonoscillatory part is
	\begin{align*}
		\frac{L_0(1-2v)}{\zeta(2-4v)}
		\frac{\pi^{2v}}{2\pi i}
		\int_{(2)}H_u(s)\mathcal I(-v,s)
		\zeta(1-2v+2s)\zeta(1-2v-2s)\,\dd s.
	\end{align*}
	The estimates used in (7.9), together with (8.1.13), dominate the truncated expressions by an $n$-summable majorant, locally uniformly for $|\Re v|<1/4$. Hence the limit $N\to\infty$ may be taken under the $s$-integral. The high-frequency cosine term is holomorphic in the same strip and is $O_{\Phi,\epsilon,A}(K^{-A})$, uniformly in the target range.
	
	The function $H_u(s)$ is odd, while $\mathcal I(-v,s)$ and the product of the two zeta factors are even in $s$. Shift the $s$-contour from $\Re s=2$ to $\Re s=-2$ through finite rectangles. The Gaussian decay of $H_u$ makes the horizontal integrals tend to $0$. By the separation assumptions on $u$ and $v$, the only poles crossed are the four distinct simple poles $s=\pm u,\pm v$. Since the integral on $\Re s=-2$ is the negative of the integral on $\Re s=2$, the latter is one half of the sum of these four residues. Consequently,
	\begin{align*}
		&\frac{\pi^{2v}}{2\pi i}
		\int_{(2)}H_u(s)\mathcal I(-v,s)
		\zeta(1-2v+2s)\zeta(1-2v-2s)\,\dd s
		\\
		&\qquad=\pi^{2v}\bigg\{
		\mathcal I(-v,u)\zeta(1+2u-2v)\zeta(1-2u-2v)
		\\
		&\qquad\qquad-\frac12H_u(-v)\mathcal I(-v,v)\zeta(1-4v)
		\bigg\}.
	\end{align*}
	The uniform Stirling formula gives
	\[
	\mathcal I(-v,s)
	=2^{-2v}K^{1-2v}\widetilde\Phi(-2v)
	+O_{\Phi,\epsilon}\left(K^{-2\Re v+\epsilon}\right),
	\qquad s\in\{u,v\}.
	\]
	It follows that
	\begin{align*}
		\mathcal R(v)
		={}&K\left(\frac{2K}{\pi}\right)^{-2v}\widetilde\Phi(-2v)
		\frac{L_0(1-2v)\zeta(1+2u-2v)\zeta(1-2u-2v)}
		{\zeta(2-4v)}
		\\
		&-\frac K2\left(\frac{2K}{\pi}\right)^{-2v}\widetilde\Phi(-2v)H_u(-v)
		\frac{L_0(1-2v)\zeta(1-4v)}{\zeta(2-4v)}
		+O_{\Phi,\epsilon}(K^\epsilon).
	\end{align*}
	
	It remains to estimate $\mathcal E_{\mathrm{reg}}(v)$. For real $t$, uniformly in the target range,
	\[
	e^{-\frac\pi2|t|}
	\left|
	\mathscr A(t)|\eta|^{-it}
	\zeta\left(\frac12+2v+it\right)
	\zeta\left(\frac12+2v-it\right)
	\mathscr P_\eta(t;v)
	\right|
	\ll_\epsilon K^\epsilon|\eta|^\epsilon(1+|t|)^{1+\epsilon}.
	\tag*{(8.2.5)}
	\]
	At $t=0$ the left-hand side is interpreted by its removable value. This estimate follows from the displayed formula for $E(i,\frac12+it)$, Stirling's formula, the convexity bounds for the two degree-one $L$-functions, the convexity bounds for the two shifted zeta factors, and the standard bound for $1/\zeta(1+it)$ on the one-line \cite[Chapters 5 and 10]{IwaniecKowalski}. 
	
	Combining (8.2.5) with (8.1.5), and choosing the decay exponent in (8.1.5) sufficiently large, gives
	\begin{align*}
		&\sum_{h\geq1}|h^{-\frac12-2v}|
		\sum_{\eta\neq0}
		\int_{\mathbb R}
		\left|
		\mathscr A(t)|\eta|^{-it}
		\mathscr D_\eta(t;v)\mathscr B_{h\eta}(t)
		\right|\,\dd t
		\\
		&\qquad\ll_{\Phi,\epsilon}K^{\frac12+\epsilon}.
	\end{align*}
	This proves both the asserted bound and the absolute, locally uniform convergence of the $h,\eta,t$ expression. The same argument with the crude strip estimate (8.1.7) proves normal convergence on compact subsets of $-1/4<\Re v<1/4$, and hence $\mathcal E_{\mathrm{reg}}(v)$ is holomorphic there. We conclude that
	\begin{align*}
		(OD)_E
		={}&K\left(\frac{2K}{\pi}\right)^{-2v}\widetilde\Phi(-2v)
		\frac{L_0(1-2v)\zeta(1+2u-2v)\zeta(1-2u-2v)}
		{\zeta(2-4v)}
		\\
		&-\frac K2\left(\frac{2K}{\pi}\right)^{-2v}\widetilde\Phi(-2v)H_u(-v)
		\frac{L_0(1-2v)\zeta(1-4v)}{\zeta(2-4v)}
		+O_{\Phi,\epsilon}\left(K^{\frac12+\epsilon}\right).
		\tag*{(8.2.6)}
	\end{align*}
	
	\subsection{Maa\ss{} spectrum}
	
	We prove that
	\[
	(OD)_M\ll_{\Phi,\epsilon}K^{\frac12+\epsilon}.
	\tag*{(8.3.1)}
	\]
	Choose a real orthonormal basis $\{u_j\}$ that simultaneously diagonalizes the Laplace operator, the Hecke operators, and the reflection operator. For
	$\mathrm{SL}_2(\mathbb Z)$, the cuspidal spectrum has no exceptional eigenvalues
	\cite[p.~261]{DeshouillersIwaniec}, so $t_j\geq0$ is real. Odd forms take the
	value $0$ at $i$, and hence only even forms need to be considered.
	
	For even forms, let $\lambda_j(m)=\rho_j(m)/\rho_j(1)$. By (8.1.4), for $\Re v>1/2$ we have, with absolute convergence,
	\begin{align*}
		(OD)_M
		={}&\frac12\sum_{j\ \mathrm{even}}\overline{\rho_j(1)}u_j(i)
		\sum_{h\geq1}h^{-\frac12-2v}
		\sum_{\eta\neq0}\mathscr B_{h\eta}(t_j)
		\sum_{q\geq1}\frac{\lambda_j(q|\eta|)}{q^{\frac12+2v}},
		\\
		\sum_{q\geq1}\frac{\lambda_j(q|\eta|)}{q^s}
		&=L(s,u_j)\prod_{p^a\parallel|\eta|}
		\left(\lambda_j(p^a)-p^{-s}\lambda_j(p^{a-1})\right).
		\tag*{(8.3.2)}
	\end{align*}
	The second identity is first obtained from the Hecke relations for $\Re s>3/2$. The trivial Hecke bound
	$|\lambda_j(m)|\leq d(m)m^{1/2}$ ensures absolute convergence and justifies rearranging the Euler product in this region.
	
	In the target region, define
	\begin{align*}
		(OD)_M
		={}&\frac12\sum_{j\ \mathrm{even}}
		\overline{\rho_j(1)}u_j(i)L\left(\frac12+2v,u_j\right)
		\sum_{h\geq1}h^{-\frac12-2v}
		\sum_{\eta\neq0}
		\prod_{p^a\parallel|\eta|}
		\left(
		\lambda_j(p^a)-p^{-\frac12-2v}\lambda_j(p^{a-1})
		\right)
		\mathscr B_{h\eta}(t_j).
		\tag*{(8.3.3)}
	\end{align*}
	This definition does not amount to replacing the divergent $q$-series term by term with an $L$-value. By the integral representation and the crude strip estimate (8.1.7), each term in (8.3.3) with fixed $j,h,\eta$ is holomorphic in $v$. By the trivial Hecke bound, for every $\Omega\Subset\{v:-1/8<\Re v<3/4\}$, we have
	\[
	\prod_{p^a\parallel m}
	\left|\lambda_j(p^a)-p^{-\frac12-2v}\lambda_j(p^{a-1})\right|
	\ll_{\Omega,\epsilon}m^{\frac12+\epsilon}.
	\]
	Combining the crude strip estimate (8.1.7), the local polynomial growth of $L(\frac12+2v,u_j)$, the standard polynomial bounds for Fourier coefficients and fixed-point values, and the local Weyl law, we find that (8.3.3) converges normally on compact subsets of this strip. It agrees with (8.3.2) for $1/2<\Re v<3/4$ and therefore, by the identity theorem, gives the continuation throughout the entire strip.
	
	In the target shift range, (8.1.5) and the trivial bound above further give
	\[
	\sum_{h\geq1}h^{-\frac12-2v}
	\sum_{\eta\neq0}
	\prod_{p^a\parallel|\eta|}
	\left(
	\lambda_j(p^a)-p^{-\frac12-2v}\lambda_j(p^{a-1})
	\right)
	\mathscr B_{h\eta}(t_j)
	\ll_{A,\Phi,\epsilon}
	K^{\frac12+\epsilon}e^{-\frac\pi2t_j}(1+t_j)^{-A}.
	\tag*{(8.3.4)}
	\]
	On the other hand, there is a fixed constant $B_0>0$ such that the
	Fourier expansion, the convexity bound in terms of the analytic conductor,
	and the local Weyl law at a fixed point
	\cite[Chapter 7]{IwaniecSpectral} give
	\begin{align*}
		e^{-\frac\pi2t_j}|\rho_j(1)|&\ll(1+t_j)^{B_0},
		\\
		L\left(\frac12+2v,u_j\right)&\ll_\epsilon K^\epsilon(1+t_j)^{\frac12+\epsilon},
		\\
		\sum_{T\leq t_j<2T}|u_j(i)|&\ll T^2,
		\qquad T\geq1.
		\tag*{(8.3.5)}
	\end{align*}
	Here the first estimate follows from the $L^2$ normalization and the Fourier expansion, while the third follows from the local Weyl law, the usual Weyl law, and the Cauchy--Schwarz inequality. The finite contribution from $t_j<1$ is absorbed directly. Decomposing the remaining spectral parameters dyadically,
	\begin{align*}
		|(OD)_M|
		&\ll K^{\frac12+\epsilon}
		\sum_{T\ \mathrm{dyadic}}(1+T)^{-A+B_0+\frac12+\epsilon}
		\sum_{T\leq t_j<2T}|u_j(i)|
		\\
		&\ll K^{\frac12+\epsilon}
		\sum_{T\ \mathrm{dyadic}}(1+T)^{-A+B_0+\frac52+\epsilon}
		\ll K^{\frac12+\epsilon},
		\tag*{(8.3.6)}
	\end{align*}
	provided that $A>B_0+4$. This proves (8.3.1).
	
	\subsection{Combining the main terms}
	
	By (5.2), (7.1), and (8.2.6), the two pairs of terms containing $H_u(v)$ and $H_u(-v)$ cancel exactly, respectively. Combining this with (4.9), (6.15), (8.1.3), and (8.3.1), we obtain
	\begin{align*}
		\mathcal S(u,v,K,\Phi)
		={}&K\left(\frac{2K}{\pi}\right)^{2u}\widetilde\Phi(2u)
		\frac{\zeta(1+2v+2u)\zeta(1-2v+2u)L_0(1+2u)}
		{\zeta(2+4u)}
		\\
		&+K\left(\frac{2K}{\pi}\right)^{-2u}\widetilde\Phi(-2u)
		\frac{\zeta(1+2v-2u)\zeta(1-2v-2u)L_0(1-2u)}
		{\zeta(2-4u)}
		\\
		&+K\left(\frac{2K}{\pi}\right)^{2v}\widetilde\Phi(2v)
		\frac{L_0(1+2v)\zeta(1+2u+2v)\zeta(1-2u+2v)}
		{\zeta(2+4v)}
		\\
		&+K\left(\frac{2K}{\pi}\right)^{-2v}\widetilde\Phi(-2v)
		\frac{L_0(1-2v)\zeta(1+2u-2v)\zeta(1-2u-2v)}
		{\zeta(2-4v)}
		\\
		&+O_{\Phi,\epsilon}\left(K^{\frac12+\epsilon}\right).
		\tag*{(8.4.1)}
	\end{align*}
	This proves (1.2), and the derivation in Section 1 then yields Corollary 1.1.

	\section{A conjecture on twisted shifted second moments}
	
	Let $\mathfrak l$ be a fixed nonzero integral ideal of $\mathbb Z[i]$, and
	put
	\[
	\alpha=u+v,\qquad \beta=u-v.
	\]
	We first describe the arithmetic factor in a form that does not require a
	choice of generator for $\mathfrak l$. For complex $\gamma$ and $\delta$,
	define the generalized divisor function on nonzero integral ideals by
	\[
	\tau_{\gamma,\delta}(\mathfrak n)
	=
	\sum_{\mathfrak a\mathfrak b=\mathfrak n}
	N(\mathfrak a)^{-\gamma}N(\mathfrak b)^{-\delta}.
	\tag*{(9.1)}
	\]
	This is a multiplicative function of $\mathfrak n$. In the domain of
	absolute convergence it gives
	\[
	L_k\left(\frac12+\gamma\right)
	L_k\left(\frac12+\delta\right)
	=
	\sum_{\mathfrak n\in\mathrm{Id}^{*}}
	\frac{\tau_{\gamma,\delta}(\mathfrak n)
		\lambda(\mathfrak n)^k}{N(\mathfrak n)^{1/2}}.
	\tag*{(9.2)}
	\]
	The diagonal arithmetic factor associated with $\mathfrak l$ is
	\[
	\mathcal A_{\mathfrak l}(\gamma,\delta)
	=
	\sum_{\substack{\mathfrak n\in\mathrm{Id}^{*}\\
			\mathfrak l\mathfrak n=(m)\text{ for some }m\geq1}}
	\frac{\tau_{\gamma,\delta}(\mathfrak n)}{N(\mathfrak n)^{1/2}}.
	\tag*{(9.3)}
	\]
	The series is initially defined for $\Re\gamma,\Re\delta>1/2$ and is
	understood near the origin by meromorphic continuation. Since every ideal
	of $\mathbb Z[i]$ is principal,
	$\lambda(\mathfrak l\mathfrak n)=1$ if and only if
	$\mathfrak l\mathfrak n=(m)$ for a positive rational integer $m$.
	
	We next separate the universal Euler product from the finite contribution
	of the twist. Write
	\[
	\mathfrak l=(q_{\mathfrak l})\mathfrak l^{\circ},
	\qquad q_{\mathfrak l}\geq1,
	\]
	where $(q_{\mathfrak l})$ is the largest ideal generated by a positive
	rational integer that divides $\mathfrak l$. Thus, if
	$(p)=\mathfrak p\overline{\mathfrak p}$ is split, then
	$\mathfrak l^{\circ}$ contains only the excess of one of the two prime-ideal
	powers; no inert prime divides $\mathfrak l^{\circ}$; and
	$\operatorname{ord}_{(1+i)}(\mathfrak l^{\circ})$ is either $0$ or $1$.
	In particular,
	$\lambda(\mathfrak l)=\lambda(\mathfrak l^{\circ})$.
	
	For every prime-ideal power
	$\mathfrak p^d\parallel\mathfrak l^{\circ}$, define
	\begin{align*}
		\mathcal P_{\mathfrak p^d}(\gamma,\delta)
		={}&
		\frac{N(\mathfrak p)^{-d/2}}
		{1+N(\mathfrak p)^{-1-\gamma-\delta}}
		\bigg(
		\tau_{\gamma,\delta}(\mathfrak p^d)
		-N(\mathfrak p)^{-1-2\gamma-2\delta}
		\tau_{\gamma,\delta}(\mathfrak p^{d-2})
		\bigg),
		\\
		\mathcal P_{\mathfrak l}(\gamma,\delta)
		={}&
		\prod_{\mathfrak p^d\parallel\mathfrak l^{\circ}}
		\mathcal P_{\mathfrak p^d}(\gamma,\delta),
		\tag*{(9.4)}
	\end{align*}
	where
	$\tau_{\gamma,\delta}(\mathfrak p^j)=0$ for $j<0$. Then the Euler product
	in (9.3) simplifies to
	\[
	\mathcal A_{\mathfrak l}(\gamma,\delta)
	=
	\frac{\zeta(1+2\gamma)\zeta(1+2\delta)
		L_0(1+\gamma+\delta)}
	{\zeta(2+2\gamma+2\delta)}
	\mathcal P_{\mathfrak l}(\gamma,\delta).
	\tag*{(9.5)}
	\]
	Indeed, at a split rational prime the local condition in (9.3) depends
	only on the difference $d$ between the exponents of
	$\mathfrak p$ and $\overline{\mathfrak p}$. If
	$x=p^{-1/2-\gamma}$, $y=p^{-1/2-\delta}$, and
	$h_m(x,y)=\sum_{a+b=m}x^ay^b$, direct summation gives
	\[
	\frac{\displaystyle\sum_{m\geq0}h_m(x,y)h_{m+d}(x,y)}
	{\displaystyle\sum_{m\geq0}h_m(x,y)^2}
	=
	\frac{h_d(x,y)-(xy)^2h_{d-2}(x,y)}{1+xy},
	\qquad d\geq1,
	\]
	with $h_{-1}=0$. This is exactly the first line of (9.4). The same formula
	with $d=1$ gives the correction at the ramified prime. Inert primes and
	prime-ideal powers occurring equally with their conjugates contribute only
	to the universal factor in (9.5).
	
\begin{conjecture}
	Suppose that $\mathfrak l$ is fixed and that $u,v$ satisfy the shift
	conditions in Theorem 1.1. Then, for every $\epsilon>0$,
	\begin{align*}
		&\sum_{k\geq1}\Phi\left(\frac{k}{K}\right)
		\lambda(\mathfrak l)^k
		L_k\left(\frac12+\alpha\right)
		L_k\left(\frac12+\beta\right)
		\\
		={}&K\widetilde\Phi(0)
		\frac{\zeta(1+2\alpha)\zeta(1+2\beta)
			L_0(1+\alpha+\beta)}
		{\zeta(2+2\alpha+2\beta)}
		\mathcal P_{\mathfrak l}(\alpha,\beta)
		\\
		&+K\left(\frac{2K}{\pi}\right)^{-2\alpha}
		\widetilde\Phi(-2\alpha)
		\frac{\zeta(1-2\alpha)\zeta(1+2\beta)
			L_0(1-\alpha+\beta)}
		{\zeta(2-2\alpha+2\beta)}
		\mathcal P_{\mathfrak l}(-\alpha,\beta)
		\\
		&+K\left(\frac{2K}{\pi}\right)^{-2\beta}
		\widetilde\Phi(-2\beta)
		\frac{\zeta(1+2\alpha)\zeta(1-2\beta)
			L_0(1+\alpha-\beta)}
		{\zeta(2+2\alpha-2\beta)}
		\mathcal P_{\mathfrak l}(\alpha,-\beta)
		\\
		&+K\left(\frac{2K}{\pi}\right)^{-2\alpha-2\beta}
		\widetilde\Phi(-2\alpha-2\beta)
		\frac{\zeta(1-2\alpha)\zeta(1-2\beta)
			L_0(1-\alpha-\beta)}
		{\zeta(2-2\alpha-2\beta)}
		\mathcal P_{\mathfrak l}(-\alpha,-\beta)
		\\
		&+O_{\mathfrak l,\Phi,\epsilon}
		\left(K^{\frac12+\epsilon}\right).
		\tag*{(9.6)}
	\end{align*}
\end{conjecture}

	At shift values where individual terms in (9.6) have poles, their sum is
	interpreted by meromorphic continuation; the polar parts are expected to
	cancel, since the moment on the left is holomorphic in the shifts.
	
	We give the heuristic for the four terms. By (9.2), the smooth $k$-average
	of the no-swap Dirichlet series contains
	$\lambda(\mathfrak l\mathfrak n)^k$. Its exact angular diagonal is
	$\lambda(\mathfrak l\mathfrak n)=1$, and hence (9.3) gives the first term
	of (9.6). The functional equation in the form
	\[
	L_k\left(\frac12+\gamma\right)
	=X_k(\gamma)L_k\left(\frac12-\gamma\right),
	\qquad
	X_k(\gamma)
	=\pi^{2\gamma}
	\frac{\Gamma(2k+\frac12-\gamma)}
	{\Gamma(2k+\frac12+\gamma)}
	\sim\left(\frac{\pi}{2k}\right)^{2\gamma}
	\tag*{(9.7)}
	\]
	allows neither factor, only the first factor, only the second factor, or
	both factors to be replaced by their dual Dirichlet series. These four
	choices replace $(\alpha,\beta)$ respectively by
	\[
	(\alpha,\beta),\qquad(-\alpha,\beta),\qquad
	(\alpha,-\beta),\qquad(-\alpha,-\beta).
	\]
	Stirling's formula in (9.7), followed by the smooth $k$-average, gives the
	corresponding powers of $2K/\pi$ and the four Mellin transforms in (9.6).
	Thus (9.6) is precisely the four-swap prediction with the
	$\mathfrak l$-dependent diagonal Euler factors retained.
	
	There are two useful consistency checks. First, when
	$\mathfrak l=(1)$, one has $\mathfrak l^{\circ}=(1)$ and
	$\mathcal P_{\mathfrak l}=1$, so (9.5) turns (9.6) exactly into Corollary
	1.1. Second, multiplying $\mathfrak l$ by an ideal generated by a rational
	integer changes neither the twist nor $\mathcal P_{\mathfrak l}$. Finally,
	for fixed $\mathfrak l$ one expects the off-diagonal to admit the same type
	of spectral treatment as in Lemma 2.2, equivalently through finitely many
	Hecke translates of the relevant incomplete Poincar\'e series. The zero
	frequency and the moving Eisenstein residues should reproduce the swapped
	terms, while the remaining Eisenstein integral and the Maa\ss{} spectrum
	should have square-root size. This explains the conjectural error term in
	(9.6); proving the required twisted spectral continuation and the uniform
	local estimates is the additional work not carried out here.

	\section{Declaration of generative AI and AI-assisted technologies in the manuscript preparation process}
	
	During the preparation of this work, the author used ChatGPT5.6 Sol (OpenAI) to translate portions of an original Chinese draft into English, assist with typesetting and formatting, and provide suggestions concerning language, stylistic and rhetorical refinement, exposition, and organization. The author reviewed and revised the AI-assisted output and takes full responsibility for the content of the manuscript.

\end{document}